\documentclass[12pt,reqno,a4paper]{amsart}

\usepackage{hyperref}
\usepackage{amssymb,amsmath, amsthm}
\usepackage[dvipsnames]{xcolor}
\usepackage{etoolbox}
\usepackage{fullpage}

\usepackage[T1]{fontenc} 
\usepackage{amsfonts}
\usepackage{enumitem}
\usepackage{mathtools}
\usepackage{bm}
\usepackage{tikz-cd}
\setenumerate{label={\normalfont(\textit{\roman*})}, ref={\textrm{\roman*}}
}
\patchcmd{\section}{\scshape}{\bfseries}{}{}
\makeatletter
\renewcommand{\@secnumfont}{\bfseries}
\makeatother

\usepackage{tikz}
\usetikzlibrary{shapes,shapes.geometric,arrows,fit,calc,positioning,arrows,automata,}

\usepackage{courier}

\allowdisplaybreaks

\usepackage{calc}
\newlength{\varwidth}
\newcommand{\shortoverline}[1]{%
  \overline{\vphantom{p}\phantom{#1}}%
}

\renewcommand{\O}{\mathcal{O}}

\DeclareRobustCommand{\SkipTocEntry}[5]{}

\newtheorem{introtheorem}{Theorem}

\newtheorem{introconjecture}[introtheorem]{Conjecture}
\newtheorem{introquestion}[introtheorem]{Question}
\theoremstyle{definition}
\newtheorem{introexample}[introtheorem]{Example}
\newtheorem{introproblem}[introtheorem]{Problem}
\newtheorem*{introdefinition*}{Definition}
\theoremstyle{plain}

\newcommand\Q{{\mathbf Q}}
\newcommand\Z{{\mathbf Z}}

\newcommand\N{{\mathbf N}}

\renewcommand{\P}{\mathbf{P}}

\theoremstyle{plain}
\newtheorem{theorem}{Theorem}[section]
\newtheorem{proposition}[theorem]{Proposition}
\newtheorem{lemma}[theorem]{Lemma}
\newtheorem{corollary}[theorem]{Corollary}

\theoremstyle{definition}

\newtheorem{remark}[theorem]{Remark}

\newtheorem{example}[theorem]{Example}
\newtheorem*{excont}{Example \continuation}
\newcommand{\continuation}{??}
\newenvironment{continueexample}[1]
 {\renewcommand{\continuation}{\ref{#1}}\excont[continued]}
 {\endexcont}

\renewcommand{\geq}{\geqslant}
\renewcommand{\leq}{\leqslant}

\newcommand{\abs}[1]{\left| #1 \right|}

\newcommand{\A}{\mathcal{A}}
\newcommand{\B}{\mathcal{B}}
\renewcommand{\P}{\mathcal{P}}

\newcommand{\MEF}{\mathrm{MEF}}

\begin{document}

\date{\today}
 \title{Automaticity of uniformly recurrent substitutive sequences}

\author[El\.zbieta Krawczyk]{El\.zbieta Krawczyk}
\address{Faculty of Mathematics\\
University of Vienna\\
Oskar Morgensternplatz 1\\
1090 Vienna, Austria \& 
Faculty of Mathematics and Computer Science\\Institute of Mathematics\\
Jagiellonian University\\
\textnormal{Stanis\l{}awa \L{}ojasiewicza 6}\\
30-348 Krak\'{o}w, Poland}
\email{ela.krawczyk7@gmail.com}

\author[Clemens M\"ullner]{Clemens M\"ullner}
\address{Institute for discrete mathematics and geometry, TU Wien, Wiedner Hauptstr. 8–10, 1040
Wien, Austria}
\curraddr{}
\email{clemens.muellner@tuwien.ac.at}


\begin{abstract} We provide a complete characterisation of automaticity of uniformly recurrent substitutive sequences in terms of the incidence matrix of the return substitution of the underlying purely substitutive sequence. This resolves a recent question posed by Allouche, Dekking and Queff\'elec in the uniformly recurrent case. We show that the same  criterion characterizes automaticity of minimal substitutive systems. 

    Furthermore, we construct a minimal substitutive system whose maximal equicontinuous factor is the 2-adic odometer, and for which the corresponding factor map is everywhere uncountable-to-one. We conjecture that a minimal substitutive system is
    $k$-automatic if and only if it is an everywhere finite-to-one extension of a $k$-adic odometer.
 \end{abstract}

\maketitle

\section*{Introduction and statement of the results}\label{sec:introduction}

 It has been observed in several contexts that certain substitutive sequences defined using substitutions of non-constant length could in fact also be obtained from substitutions of constant length. While it is easy to construct such examples artificially, they also occur naturally, and the corresponding constant-length substitution is often by no means obvious. Such discoveries of 'hidden automatic sequences' (a name we borrow from \cite{hidden-20}) are often insightful since automatic sequences are considerably better understood  and can be treated using more specialized tools  (e.g.\  finite automata).  A particularly striking example is the Lys\"{e}nok morphism related to the presentation of the first Grigorchuk group \cite{GLN-18}, where spectral properties of the system generated by the Lys\"{e}nok morphism are used to deduce spectral properties of the Schreier graph of the Grigorchuk group.   
 
 In the opposite direction, a problem of showing that a given substitutive sequence is not automatic has also appeared in several contexts, e.g.\ in the study of gaps between factors in the famous Thue--Morse sequence \cite{Spieg-21} or in the mathematical description of  the drawing of the classical  Indian kolam \cite{AAS-06}. In each case, some \emph{ad hoc} methods are employed to prove or disprove the automaticity of the substitutive sequence under consideration.

The problem of how to recognize that a substitutive sequence is automatic  has been raised recently in \cite{hidden-20} by Allouche, Dekking and Queff\'elec, and we refer the reader there for other interesting examples. 
    However, the problem  was already explicitly stated in 2003 in the classical book of Allouche and
Shallit  \cite[Prob. 3, Sec. 7.11]{AlloucheShallit-book}.
    It is natural to also consider the dynamical version of the problem.

\begin{introproblem}\cite{hidden-20}\label{prob} For a given substitutive sequence $x$, decide whether it is automatic. For a given system $X=\overline{\mathcal{O}(x)}$ generated by a substitutive sequence $x$ decide whether it is automatic (i.e.\ whether it can be generated by an automatic sequence).
\end{introproblem}

The purpose of this paper is to solve Problem \ref{prob} in the case when $x$ is uniformly recurrent. It turns out that the answer does not depend on which version of the problem one considers and, indeed, dynamical methods play an important role in the resolution of the original question. Since any periodic sequence is $k$-automatic for any $k\geq 2$, we restrict ourselves to the case when $x$ is nonperiodic (equivalently, $X$ is infinite).

Formally, a (one-sided or two-sided) sequence $x$ is called \emph{substitutive} if it is a coding of a fixed point of some substitution $\varphi\colon \mathcal{A}\to\mathcal{A}^*$ over a finite alphabet $\mathcal{A}$.
If one can take the substitution $\varphi$ to be of constant length $k$, then $x$ is called $k$-\emph{automatic}.
A (one sided or two-sided) symbolic system $X$ is called \emph{substitutive} (resp.\ $k$-\emph{automatic}) if there exists a substitutive (resp.\ $k$-automatic) sequence $x$ such that $X=\overline{\mathcal{O}(x)}$.

For a primitive substitution $\varphi\colon \mathcal{A}\to\mathcal{A}^*$,  let $x$ be a (one-sided or admissible two-sided) fixed point of $\varphi$ and let $a=x_0$. (A two-sided fixed point $x$ of $\varphi$ is \textit{admissible} if  the word $x_{-1}x_0$ appears in $\varphi^n(a)$ for some $a\in\mathcal{A}$ and $n\geq 1$.)
 Let $\mathrm{R}_a$ be the set of return words to $a$ in $x$, that is, the set of words  $w\in\mathrm{L}(x)$ such that  $w$ starts with $a$, $w$ has exactly one occurrence of $a$, and $wa\in\mathrm{L}(x)$. Since $\varphi$ is primitive, $\mathrm{R}_a$ is finite and we may define in the natural way   the return substitution $\tau\colon\mathrm{R}_a\to\mathrm{R}_a^*$ of $\varphi$ to $a$. (See Section \ref{sec:proof} for precise definitions; for an example, see  Example \ref{ex:return} below). Our main result is the following theorem.

\begin{introtheorem}\label{thm:automatic} Let $\varphi\colon \mathcal{A}\to\mathcal{A}^*$ be a primitive substitution, let $x$ be a one-sided or an admissible two-sided fixed point of $\varphi$, assume that $x$ is nonperiodic, and let $a=x_0$. Let $X$ be the (one-sided or two-sided) system generated by $x$.  Let $\tau\colon\mathrm{R}_a\to\mathrm{R}_a^*$ be the return substitution to $a$, let $M_{\tau}$ denote  the incidence matrix of $\tau$, and let $s\geq 0$ denote the size of the largest Jordan block of $M_{\tau}$ corresponding to the eigenvalue $0$.  The following conditions are equivalent:\begin{enumerate}
\item\label{thm:automatic0} $x$ is an automatic sequence;
\item\label{thm:automatic0.5} there is a coding $\varrho\colon \A\to\B$ such that $y=\varrho(x)$ is nonperiodic and automatic;
\item\label{thm:automatic3} $X$ is an  automatic system;
\item\label{thm:automatic1} $X$ has an infinite automatic system $Y$ as a (topological) factor;
\item\label{thm:automatic2} ${}^{t}(|\varphi^s(w)|)_{w\in\mathrm{R_{a}}}$ is a left eigenvector of $M_{\tau}$.
\end{enumerate}
\end{introtheorem}

\begin{introexample}\label{ex:return}Let $\varphi\colon\mathcal{A}\to\mathcal{A}^*$ be a primitive substitution given by 
\[a\mapsto aca,\ b\mapsto bca, c\mapsto cbcac,\]
and let $x=acac\dots$ be a (one-sided) fixed point of $\varphi$ starting with $a$.   The set of return words to $a$ in $x$ is given by $\mathrm{R}_a=\{ac, acbc\}$. To see this, note that $ac$ is the first return word to $a$ occurring in $x$. The word $\varphi(ac)=ac|acbc|ac$ is a concatenation of 3 return words to $a$ in which $acbc$ is the only new word. Applying $\varphi$ to it, we see that $\varphi(acbc)=ac|acbc|acbc|acbc|ac$ is a concatenation of $5$ return words and no new return words appear in this factorisation. Hence $\mathrm{R}_a$ consists exactly of these two words. Relabelling, $1=ac$, $2=acbc$, we get that the return substitution $\tau\colon \{1,2\}\to \{1,2\}^*$ is given by
\[1\mapsto 121, \ 2\mapsto 12221.\]

It is easy to check that the fixed point $x$ of $\varphi$ is not periodic. The incidence matrix of the return substitution $\tau\colon \{1,2\}\to \{1,2\}^*$ is given by 
\[M_{\tau}= \begin{pmatrix}2 & 2 \\
1 & 3
\end{pmatrix},\]
and has eigenvalues 4 and 1; in particular, s=0. Since ${}^{t}(|w|)_{w\in\mathrm{R_{a}}}=(2,4)$ is a left eigenvector of $M_{\tau}$ corresponding to the eigenvalue 4, by Theorem \ref{thm:automatic}, the fixed point $x$ of $\varphi$ is automatic (more precisely, it is $4$-automatic, and, hence, $2$-automatic). 
 \end{introexample}

Theorem \ref{thm:automatic} settles Problem \ref{prob} for any nonperiodic uniformly recurrent substitutive sequence $y$ since any such sequence is given as a coding $\varrho$ of a fixed point $x$ of some primitive substitution, which can be found algorithmically, see  \cite[Thm.\ 3]{D-recurrence}, or \cite[Sec.\ 3]{DL-2018} (see also \cite[Thm.\ ~2.1]{MR} for a related result); in this case the coding $\varrho$ induces also a natural factor map between the systems $X$ and $Y$ generated by $x$ and $y$, respectively. 
The proof is algorithmic; if condition \eqref{thm:automatic2} is satisfied, it yields a construction of a constant length substitution and a coding generating $x$.

 Let $x$ be a substitutive sequence generated by some substitution $\varphi$ and coding $\tau$. A well-known necessary condition for  $x$  to be automatic  comes from a version of Cobham's theorem proven by Durand \cite{D-cobham}: if $x$ is not ultimately periodic and $k$-automatic, then the dominant eigenvalue of $M_{\varphi}$ is multiplicatively dependent with $k$.  This condition is of course far from being sufficient: there are many primitive substitutions whose dominant eigenvalue is an integer and whose nonperiodic fixed points are not automatic, see e.g.\ Example \ref{ex:gaps}, \ref{ex:kolam}, or \ref{ex:spectral} below. 
    In the opposite direction, a  useful sufficient condition for $x$ to be automatic was obtained by Dekking in 1976: $x$ is automatic if the length vector $^{t}(|\varphi(a)|)_{a\in\mathcal{A}}$ is a left eigenvector of $M_{\varphi}$ \cite{DK-77}, see also \cite{hidden-20}. In Corollary \ref{cor:left-proper} below we show that for a special class of \emph{left-proper} substitutions Dekking's criterion is essentially an if and only if condition. A recent paper by Allouche, Shallit and Yassawi \cite{ASY-21} provides a handy toolkit of methods of showing that a (substitutive) sequence is not automatic; nevertheless, prior to Theorem \ref{thm:automatic} no general necessary and sufficient conditions were known.

    At the same time, ~\cite[Section 10]{ASY-21} raises an interesting point: one can often detect that a uniformly recurrent substitutive sequence $x$ is not automatic by analysing the maximal equicontinuous factor ($\MEF$) of the dynamical system $X$ generated by $x$. In particular, if $X$ admits a dynamical eigenvalue which is not a root of unity, then $x$ cannot be automatic. This criterion is again not necessary: there exist minimal non-automatic substitutive systems whose $\MEF$ is a $k$-adic odometer; see Example~\ref{ex:spectral}.

Nevertheless, the question of whether automaticity of substitutive systems can be characterised by purely dynamical properties appears to be a natural and interesting one. In fact, the non-automaticity of the system $X$ in Example~\ref{ex:spectral} \emph{can} be deduced by purely dynamical means. We show in Section~\ref{sec:spectral} that a factor map
$\pi \colon X\to \Z_2$ to the $\MEF$ $\Z_2$ of $X$ is \emph{everywhere} infinite-to-one, whereas it is well known that every $k$-automatic system is a finite-to-one extension of a $k$-adic odometer \cite{DK-77}. As far as we are aware, this provides the first such example in the literature---likely because fibre properties of the $\MEF$ have traditionally been studied only for constant-length or Pisot-type symbolic substitution systems.

\begin{introtheorem}\label{intothm:infinitefibres}
Let $\varphi\colon \{a,b\}\to\{a,b\}^*$ be a substitution given by
\[a\mapsto abbbba,\ b\mapsto aa\]
and let $X$ be the system generated by the unique two-sided fixed point of $\varphi$. Then $X$ has $\Z_2$ as the $\MEF$ and all the fibres of a factor map $\pi\colon X\to \Z_2$ are uncountable, that is, $\pi^{-1}(z)\subset X$ is uncountable for any $z\in\Z_2$.    
\end{introtheorem}

In fact, we believe that this phenomenon points towards a genuine dynamical characterisation of automaticity for minimal substitutive systems. More precisely, we propose the following conjecture; see Section~3 for further motivation and discussion. For simplicity we state it only in the two-sided case.

  \begin{introconjecture}\label{con:dynamic_char}
Let $\varphi$ be a primitive substitution with an admissible  fixed point $x$. Assume that the dominant eigenvalue of $M_{\varphi}$ is an integer $k$ and that $X=\overline{\mathcal{O}(x)}$ is infinite.
	The following two conditions are equivalent:
\begin{enumerate}
\item\label{con:dynamic_char1} $x$ is $k$-automatic,
\item\label{con:dynamic_char2} $\Z_k$ is a factor of $X$ and all the fibres of a factor map $\pi\colon X\to\Z_k$ are finite; i.e.\ $\pi^{-1}(z)\subset X$ is finite for all $z\in \Z_k$.
\end{enumerate}
\end{introconjecture}
We do not know whether the assumption that the dominant eigenvalue of $M_{\varphi}$ is an integer is necessary in Conjecture \ref{con:dynamic_char}.

\begin{introquestion}\label{que:dynamic_char}
Let $\varphi$ be a primitive substitution with an admissible  fixed point $x$. Assume $X=\overline{\mathcal{O}(x)}$ is infinite.
	Are the following two conditions equivalent:
\begin{enumerate}
\item\label{que:dynamic_char1} $x$ is automatic,
\item\label{que:dynamic_char2} there exists $k\geq 2$ such that $\Z_k$ is a factor of $X$ and all the fibres of a factor map $\pi\colon X\to\Z_k$ are finite?
\end{enumerate}
\end{introquestion}
	
	In other words, Question \ref{que:dynamic_char} (versus Conjecture \ref{con:dynamic_char}) asks if there exists a primitive substitution $\varphi$ whose dominant eigenvalue is \emph{not} an integer, yet $X_{\varphi}$ has everywhere finite-to-one factor map to $\Z_k$ for some $k\geq 2$.
Instead of requiring all fibres to be finite, one could equivalently stipulate in~\eqref{con:dynamic_char2} that the cardinalities of the fibres are uniformly bounded; i.e.,  $\abs{\pi^{-1}(z)}<K$ for all $z\in\Z_k$ for some independent constant $K>0$. A stronger version of Conjecture \ref{con:dynamic_char} (or Question \ref{que:dynamic_char}) would ask whether, for a non-automatic fixed point $x$, \emph{all} fibres of the factor map $\pi\colon X\to\Z_k$ are necessarily infinite.

To place Conjecture~\ref{con:dynamic_char} and Question \ref{que:dynamic_char} in a broader context, we note that the structure of fibres over the $\MEF$ is by now quite well understood for minimal constant-length  and  Pisot-type symbolic substitution systems (in the latter case, modulo the well-known Pisot Substitution Conjecture). It is also quite well understood when a general minimal substitution system is weakly mixing, that is, when its $\MEF$ is trivial. Outside of these settings, however, very little seems to be known about the (generic) behaviour of fibres of the $\MEF$. Focusing first on substitution systems whose $\MEF$ is an odometer appears to be a natural and tractable starting point; a positive answer to Question~\ref{que:dynamic_char} would provide a satisfactory answer in this case.

\subsection*{Acknowledgments} 
We wish to thank Jean-Paul Allouche for his interest in our work and pointing to us some relevant references. 
    The first author was supported by National Science Center, Poland under grant no. UMO2021/41/N/ST1/04295. 
For the purpose of Open Access, the author has applied a CC-BY public copyright licence to any Author Accepted Manuscript (AAM) version arising from this submission. 
    The second author was supported by the FWF (Austrian Science Fund), project F5502-N26, which is a part of the Special Research Program ``Quasi Monte Carlo methods: Theory and Applications'', and by the project ArithRand, which is a joint project between the ANR (Agence Nationale de la Recherche) and the FWF, grant numbers ANR-20-CE91-0006 and I4945-N.

\section{Examples and corollaries}

We say that a substitution  $\varphi\colon \mathcal{A}\to\mathcal{A}^*$ is \textit{left-proper} if all words $\varphi(a)$, $a\in\mathcal{A}$, share the same initial symbol. 
 For a substitutive sequence generated by a \textit{left-proper} substitution the criterion for automaticity takes the following simplified form not requiring the computation of the return words (see the end of Section \ref{sec:proof} for the proof).
  
\begin{corollary}\label{cor:left-proper} Let $\varphi\colon \mathcal{A}\to\mathcal{A}^*$ be a primitive, left-proper substitution,   let $M_{\varphi}$ denote  the incidence matrix of $\varphi$, and let $s$ denote the size of the largest Jordan block of $M_{\varphi}$ corresponding to the eigenvalue $0$. Let $\tau\colon \mathcal{A}\to\mathcal{B}$ be a coding, let $x$ be a one-sided or an admissible two-sided fixed point of $\varphi$, let $y=\tau(x)$ and assume that $y$ is not periodic.  The following conditions are equivalent:\begin{enumerate}
\item\label{cor:proper1} $y$ is automatic;
\item\label{cor:proper2} $^{t}(|\varphi^s(a)|)_{a\in\mathcal{A}}$ is a left eigenvector of $M_{\varphi}$.
\end{enumerate}
\end{corollary}

We should note that any system $X$ generated by an admissible fixed point of a primitive substitution can be, in fact, obtained from some primitive and \textit{left-proper} substitution by an  algorithmic procedure \cite[Prop.\ 31]{D-2000}. This process consists of two steps: first, one computes the set $\mathrm{R}_a$ of return words to $a=x_0$ and the return substitution $\tau\colon\mathrm{R}_a\to\mathrm{R}_a^*$; second, one considers the alphabet $\mathcal{B}=\{(w,i)\mid w\in \mathrm{R}_a, 0\leq i<|\varphi(w)|\}$ and defines a new substitution $\zeta\colon \mathcal{B}\to\mathcal{B}^*$, which is left-proper and gives rise to the system conjugate with $X$ (see \cite{D-2000}  for the definition of $\zeta$).  In view of this, Corollary \ref{cor:left-proper} can also be treated as an (algorithmic) solution to Problem \ref{prob}. However, the second step in this 'properisation'   process  greatly increases the size of the matrix, and, in view of Theorem \ref{thm:automatic} is not necessary for the solution of Problem \ref{prob}. We illustrate Corollary \ref{cor:left-proper} with the following examples from the literature.

\begin{example}\label{ex:gaps}\cite{Spieg-21}  Let $\mathcal{A}=\{a,\overline{a},b,c\}$, let $\psi\colon \mathcal{A}\to\mathcal{A}^*$ be a  substitution given by 
\[a\mapsto a\overline{a}bc, \ \overline{a}\mapsto a\overline{a}cb, \ b\mapsto a\overline{a}bcb,\  c\mapsto a\overline{a}c,\]
and let $\rho\colon \mathcal{A}\to \{2,3,4\}$ be the coding given by 
\[a\mapsto 3, \ \overline{a}\mapsto 3, \ b\mapsto 4, \ c\mapsto 2.\]
Let $B=33423\dots$ be the coding by $\rho$ of the (unique) one-sided fixed point of $\psi$. The sequence  $B$ encodes the differences of the consecutive occurrences of the word $01$ in the famous Thue--Morse sequence \cite[Lem.\ 3]{Spieg-21}  (note that we have taken here the second power of the substitution considered in  \cite{Spieg-21},   so that $\psi$ is left-proper). Sequence $B$ has been recently analysed  by Spiegelhofer, who showed (among other things) that $B$ is not automatic using the kernel-based characterisation of automaticity  \cite[Thm.\ 1]{Spieg-21}.   We will show that $B$ is not automatic using Corollary \ref{cor:left-proper}. It is easy to see that $B$ is not periodic. The eigenvalues of the incidence matrix $M_{\psi}$ are given by 4,1,0,0, and $M_{\psi}$ has two simple Jordan blocks corresponding to the eigenvalue 0, so  $s=1$. Since $\psi$ is primitive and left-proper, and
\[^{t}(|\psi(a)|)_{a\in\mathcal{A}}M_{\psi}=(16,16,21,11)\neq (16,16,20,12)={}^{t}(|\psi(a)|)_{a\in\mathcal{A}}\cdot 4,\]
the sequence $B$ is not automatic.
\end{example}

\begin{corollary}\label{cor:singular} Let $\varphi\colon\mathcal{A}\to\mathcal{A}^*$ be a left-proper, primitive substitution and assume that the incidence matrix $M_{\varphi}$ is not singular. Then, a nonperiodic (one-sided or admissible two-sided) fixed point of $\varphi$ is automatic if and only if the substitution $\varphi$ is of constant length.
\end{corollary}
\begin{proof}    By Corollary \ref{cor:left-proper}, $x$ is automatic if and only if the horizontal vector consisting of 1's is a left eigenvector of $M_{\varphi}$  (since $s=0$ in this case), which happens if and only if $\varphi$ is of constant length.
\end{proof}

\begin{example}\label{ex:kolam}\cite{AAS-06} Let  $x=GDDGG\dots$ be the fixed point of the  substitution
\[\lambda(G)=GDD,\ \lambda(D)=G.\]
and note that $x$ is nonperiodic. The sequence $x$  occurs in the drawings of the classical Indian kolam and has been analysed in \cite{AAS-06}, where the authors showed (among other things) that $x$ is not automatic using the kernel-based characterisation of automaticity \cite[Thm.\  3.1]{AAS-06}.  Since $\det(M_{\lambda})=-2$ and $\lambda$ is left-proper, primitive and not of constant length it follows immediately from Corollary \ref{cor:singular} that $x$ is not automatic.
\end{example}

For general nonproper  primitive substitutions, there is no criterion for automaticity that depends only on the incidence matrix of the substitution as the following example shows.

\begin{example}\label{ex:matrix} Let $\varphi\colon \{a,b,c\}\to\{a,b,c\}^*$ be the substitution considered in Example \ref{ex:return}, i.e.\
\[a\mapsto aca,\ b\mapsto bca, c\mapsto cbcac,\]
and recall that the fixed point $acabacabacababac\ldots$ of $\varphi$  is $2$-automatic.  Now, consider the (unique) fixed point $x$ of the substitution $\varphi'\colon \{a,b,c\}\to\{a,b,c\}^*$ given by
\[a\mapsto aca,\ b\mapsto acb, c\mapsto abccc,\]
which has the same incidence  matrix as $\varphi$. Then, $x$ is not periodic and, by Corollary \ref{cor:singular}, is not automatic, since $\varphi'$ is left-proper and $M_{\varphi'}$ is not singular.
\end{example}

Let $\varphi\colon\mathcal{A}\to\mathcal{A}^*$ be a (left-proper) substitution and let $s$ be the size of the largest Jordan block of $M_{\varphi}$ corresponding to the eigenvalue $0$. Note that if $^{t}(|\varphi^s(a)|)_{a\in\mathcal{A}}$ is a left eigenvector of $M_{\varphi}$, then ${}^{t}(|\varphi^n(a)|)_{a\in\mathcal{A}}$ is a left eigenvector of $M_{\varphi}$ for any $n\geq s$. It can happen that $^{t}(|\varphi^n(a)|)_{a\in\mathcal{A}}$ is a left eigenvector of $M_{\varphi}$ for some $n<s$, consider e.g.\ the substitution
\[0\mapsto 010, \quad 1\mapsto 001,\] which is  of constant length (and so $n=0$ works) and for which $s=1$.  Nevertheless, the following example shows that, in general, the integer $s$ in Corollary \ref{cor:left-proper} is optimal.

\begin{example}\label{ex:optimal} Let $\varphi\colon \mathcal{A}\to\mathcal{A}^*$ be any left-proper (primitive) substitution on the three letter alphabet  with the incidence matrix 
\[M_{\varphi}=\begin{pmatrix}
4 & 3& 1\\
4 & 1 & 3\\
4 &1 &3
\end{pmatrix}\] such that the (unique) one-sided fixed point of $\varphi$ is not periodic, e.g.\ let $\varphi$ be given by the formula
\[a\mapsto aaaabbbbcccc,\quad b\mapsto abcaa, \quad c\mapsto abbbccc.\] 
It is easy to check that the fixed point $x=aaaab\dots$ of $\varphi$ is not periodic. The eigenvalues of $M_{\varphi}$ are 8 and 0, and $M_{\varphi}$ has a Jordan block of size 2 corresponding to the eigenvalue 0.  The vector $^{t}(|\varphi^2(a)|)_{a\in\mathcal{A}}=(94, 48, 48)$ is a left eigenvector of $M_{\varphi}$ (corresponding to the eigenvalue 8) and, by Corollary \ref{cor:left-proper}, $x$ is automatic. However, neither $(1,1,1)$ nor $^{t}(|\varphi(a)|)_{a\in\mathcal{A}}=(12,5,7)$ is a left eigenvector of $M_{\varphi}$.
\end{example}

Recall that a complex number $\lambda$ is a (topological) dynamical eigenvalue of a subshift $X$ if there exists a continuous function $f\colon X\to\mathbf{C}$ such that $f\circ T= \lambda f$, where $T$ denotes the shift map. Thanks to the earlier work of Dekking~\cite{DK-77} and the recent work of the second author and Yassawi~\cite{MY-21}, the dynamical eigenvalues of minimal automatic systems are well understood.  For an infinite minimal $k$-automatic system $X$ its eigenvalues are given by $k^n$-th roots of unity, $n\geq 1$ and $h$-th roots of unity, where $h$ is an integer coprime with $k$ known as the \textit{height} of $X$. This is the same as saying that the additive group $\Z_k\times \Z/h\Z $ is the maximal equicontinuous factor of $X$, where $\Z_k$ is the ring of $k$-adic integers (we refer to ~\cite{MY-21} for more details). One may often show that a given (minimal) substitutive system is not automatic by computing its (dynamical) eigenvalues as described e.g.\ in \cite[Sec.\ 10]{ASY-21}.  However, there exist minimal substitutive systems that have $\Z_k$ as the maximal equicontinuous factor and are not automatic.

\begin{example}\label{ex:spectral} Let $\varphi\colon \{a,b\}\to\{a,b\}^*$ be a primitive substitution given by
\[a\mapsto abbbba,\ b\mapsto aa,\]
let $x=abbbbaaa\dots$ be the (unique) one-sided fixed point of $\varphi$, and let $X$ be the orbit closure of $x$. We will compute the dynamical eigenvalues of $X$. It follows from \cite[Cor.\ 1]{FMN-96}, that a  system given by a primitive substitution, whose incidence matrix has only integer eigenvalues, cannot have irrational dynamical eigenvalues (i.e.\ eigenvalues $e^{2\pi i\alpha}$ with $\alpha\notin \Q$). Since the eigenvalues of $M_{\varphi}$ are given by $-2$, and $4$, $X$ has no irrational (dynamical) eigenvalues. On the other hand, by  \cite[Prop.\ 2]{FMN-96}  (see also \cite[Lem.\ 28]{D-2000}), $e^{ \frac{2\pi ip}{q}}$ is an eigenvalue of $X$ for some $p\in \Z$, $q\geq 1$ if and only if $q$ divides both $|\varphi^n(a)|$ and $|\varphi^n(b)|$ for some $n\geq 0$. An easy computation shows that
\begin{equation*}
|\varphi^n(a)|=-\frac{1}{3}(-2)^n + \frac{4}{3}\cdot 4^n, \quad
|\varphi^n(b)|=\frac{1}{3}(-2)^n + \frac{2}{3}\cdot 4^n,
\end{equation*}
for all $n\geq 0$. Hence, $e^{\frac{2\pi i k}{2^m}}$ is an eigenvalue of $X$ for all $k\in \Z$ and $m\geq 1$. It is easy to see that  no prime $p$ other than 2 divides $\gcd(|\varphi^n(a)|,|\varphi^n(b)|)$  for any $n\geq 0$ (since $M_{\varphi}$ is invertible modulo $p$ and 
$(1,1)M^n_{\varphi}=(|\varphi^n(a)|,|\varphi^n(b)|)$ for all $n\geq 0$). Thus, dynamical eigenvalues of $X$ are given exactly by $2^m$-th roots of unity, $m\geq 1$, or, equivalently, $\Z_2$ is the maximal equicontinuous factor of $X$. Nevertheless, by Corollary \ref{cor:singular}, $x$ is not automatic, since $\varphi$ is left-proper, $M_{\varphi}$ is not singular, and $\varphi$ is not of constant length.
\end{example}

In Section \ref{sec:spectral} we will continue to study properties of the substitution $\varphi$ from Example \ref{ex:spectral} and prove Theorem \ref{intothm:infinitefibres} from the introduction.

\section{Proof of Theorem  \ref{thm:automatic} }\label{sec:proof}

The following section is devoted to the proof of Theorem \ref{thm:automatic}. First, we fix our notation and recall some standard notions.

\subsection*{Words and sequences}
 Let $\mathcal{A}$ be a finite set (called an \emph{alphabet}).  We let $\mathcal{A}^*$ denote the set of finite words over $\mathcal{A}$, and  $\mathcal{A}^+$  the set of nonempty finite words over $\mathcal{A}$.   We let $\mathcal{A}^{\mathbf{N}}$ denote the set of infinite sequences over $\mathcal{A}$, where $\mathbf{N}=\{0,1,2,\ldots\}$ stands for the set of nonnegative integers, and $\mathcal{A}^{\mathbf{Z}}$ the set of biinfinite (or two-sided) sequences.  For a word $u\in\mathcal{A}^*$, we let $|u|$ denote the length of $u$. All finite words  are indexed starting at $0$. For a sequence or a finite word $x$  and integers $i\leq j$ we write  $x_{[i,\,j)}$ for the word $x_ix_{i+1}\cdots x_{j-1}$, $x_{[i,\,\infty)}$ for the infinite sequence $x_ix_{i+1}\cdots$ and $x_{(-\infty,i]}$ for the left-infinite sequence $\ldots x_{i-1}x_i$, when these make sense. We say that a word $u$  \textit{appears} in $x$ at position $i$ if  $u=x_{[i,j)}$ for some $j$.

\subsection*{Substitutive sequences} Let $\mathcal{A}$ and $\mathcal{B}$ be alphabets. A \textit{morphism} is a map $\varphi\colon \mathcal{A}\rightarrow \mathcal{B}^*$ that assigns to each letter $a\in\mathcal{A}$ some finite word $w$ in $\mathcal{B}^*$. A morphism $\varphi$ is \textit{nonerasing} if $|\varphi(a)|\geq 1$ for all $a\in\mathcal{A}$. A morphism  $\varphi$ is \textit{of constant length} $k$ if $|\varphi(a)|=k$ for each $a\in \mathcal{A}$.  A \textit{coding} is a morphism of constant length 1, i.e.\ an arbitrary map $\tau\colon \mathcal{A}\rightarrow \mathcal{B}$. If $\mathcal{A}=\mathcal{B}$, we refer to any morphism  $\varphi$ as \textit{substitution}. A morphism $\varphi\colon \mathcal{A}\rightarrow \mathcal{B}^*$ induces natural maps $\varphi\colon \mathcal{A}^{\mathbf{N}}\to \mathcal{B}^{\mathbf{N}}$ and $\varphi\colon \mathcal{A}^{\mathbf{Z}}\to \mathcal{B}^{\mathbf{Z}}$; in the latter case, the map is given by the formula
\[\varphi(\ldots x_{-1}.x_{0}\ldots)=\ldots \varphi(x_{-1}).\varphi(x_0)\ldots,\] where the dot indicates the 0th position.  For a substitution $\varphi\colon \mathcal{A}\rightarrow \mathcal{A}^*$, a sequence $x$ in $\mathcal{A}^{\mathbf{N}}$ or in $\mathcal{A}^{\mathbf{Z}}$ is called a \textit{fixed point of} $\varphi$ if $\varphi(x)=x$. A two sided fixed point $x$ is said to be \textit{admissible} if the word $x_{-1}x_0$ appears in $\varphi^n(a)$ for some $a\in\mathcal{A}$ and $n\geq 1$.

Let $k\geq 2$. A (one-sided or two-sided) fixed point of a substitution (resp.\  substitution of constant length $k$) is called a \textit{purely substitutive} (resp.\  \textit{purely}  $k$-\textit{automatic}) sequence. A sequence is called \textit{substitutive} (resp.\ $k$-\textit{automatic}) if it can be obtained as the image of a purely substitutive (resp.\ purely $k$-automatic) sequence under a coding.  It is easy to see that a two-sided sequence $(x_n)_{n\in\Z}$ is substitutive (resp.\ $k$-automatic) if and only if the one-sided sequences $(z_n)_{n\geq 0}$ and $(z_{-n})_{n<0}$  are substitutive (resp.\ $k$-automatic); this follows e.g.\ from (the proof of) \cite[Lemma 2.10]{BKK}. For all $n\geq 1$, a sequence $x$ is $k$-automatic if and only if it is $k^n$-automatic \cite[Theorem 6.6.4]{AlloucheShallit-book}, and that  all periodic  sequences are $k$-automatic with respect to any $k\geq 2$ \cite[Thm.\ 5.4.2]{AlloucheShallit-book}.

\subsection*{Substitutive systems} Let $\mathcal{A}$ be an alphabet. The set $\mathcal{A}^{\mathbf{Z}}$  with the product topology (where we use the discrete topology on each copy of $\mathcal{A}$) is a compact metrisable space. We define the shift map $T\colon \mathcal{A}^{\mathbf{Z}}\rightarrow \mathcal{A}^{\mathbf{Z}}$ by $T((x_n)_n)= (x_{n+1})_n$. A set $X\subset \mathcal{A}^{\mathbf{Z}}$ is called a \textit{subshift} if $X$ is closed and $T(X)\subset X$.  We let $\mathrm{L}(X)$  denote the \textit{language} of the subshift $X$, i.e.\ the set of all finite words which appear in some $x\in X$, and  $\mathrm{L}^r(X)$ denote the set of words of length $r$ which belong to the language of $X$.  We also use $\mathrm{L}(x)$ (resp.\ $\mathrm{L}^r(x)$) to denote the set of words (resp.\  set of words of length $r$) which appear in a sequence $x$. A nonempty subshift $X$ is \textit{minimal} if it does not contain any  subshifts other than $\emptyset$ and $X$. Equivalently, $X$ is minimal if and only if each point $x\in X$ has a dense orbit in $X$, and if and only if each sequence $x\in X$ is  \emph{uniformly recurrent}, i.e.\ every word that appears in $x$ does so with bounded gaps.  A subshift $Y$ is a \textit{(topological) factor} of the subshift $X$ if there exists a continuous surjective map $\pi\colon X\rightarrow Y$, which commutes with the shift map $T$. Such a map $\pi$ is called \textit{a factor map}. Two subshifts $X$ and $Y$ are \textit{conjugate} (or \textit{isomorphic}) if there exists a homeomorphism $\pi\colon X\rightarrow Y$, which commutes with $T$. In what follows, we will also work with one-sided subshifts $X\subset \mathcal{A}^{\N}$; all  definitions can be adapted to this setting in a straightforward way.

With every one-sided or two-sided substitutive sequence $x$, we associate a subshift given by the orbit closure of $x$. Formally,  a system $X\subseteq \mathcal{A}^{\Z}$ is called \emph{substitutive} (resp.\ \emph{purely substitutive}, $k$-\emph{automatic}, \emph{purely} $k$-\emph{automatic}) if there exists a  substitutive (resp.\ purely substitutive, $k$-automatic, purely $k$-automatic) sequence $x\in\mathcal{A}^{\Z}$ such that $X=\overline{\mathcal{O}(x)}=\overline{\{T^n(x)\mid n\in \Z\}}$. Similarly, a system $X\subseteq \mathcal{A}^{\N}$ is called \emph{substitutive} (resp.\ \emph{purely substitutive}, $k$-\emph{automatic}, \emph{purely} $k$-\emph{automatic}) if there exists a purely substitutive (resp.\ substitutive, $k$-automatic, purely $k$-automatic) sequence $x\in\mathcal{A}^{\N}$ such that $X=\overline{\mathcal{O}^{+}(x)}=\overline{\{T^n(x)\mid n\in \N\}}$.

In this paper, we will be mostly interested in \textit{minimal} substitutive systems. With every substitution $\varphi\colon \mathcal{A}\to\mathcal{A}^*$ one can associate its \textit{incidence matrix}  indexed by $\mathcal{A}$ and defined as $M_{\varphi}= (|\varphi(b)|_a)_{a,b\in \mathcal{A}}$, where $|\varphi(b)|_a$ denotes the number of occurrences of the letter $a$ in $\varphi(b)$. The  equation $M_{\varphi^{n}}=M^n_{\varphi}$ is satisfied for all $n\geq 1$. The substitution $\varphi$ is called \textit{primitive} if the matrix $M_{\varphi}$ is primitive, i.e.\ there exists $n\geq 1$ such that all entries of $M_{\varphi}^n$ are strictly positive. If $\varphi$ is primitive, then there exists some power $\varphi^n$ of $\varphi$ (with $n<|\mathcal{A}|^2)$ such that $\varphi^n$ admits an admissible two-sided fixed point (and, hence, also a one-sided fixed point). Thus, without loss of generality, we may assume that all primitive substitutions $\varphi$ admit at least one admissible fixed-point $x$. In this case the subshift generated by $x$ is minimal \cite[Prop.\ 5.5]{Queffelec-book}.

\subsection*{Factorisations} Let $\mathcal{A}$ be an alphabet and let  $W\subset \mathcal{A}^+$. A word, one-sided sequence or two-sided sequence $x$ is \textit{factorizable} over $W$ if $x$ can be written as a concatenation of words in $W$. In this case, a $W$-factorisation of a one-sided sequence or a word $x$ is, respectively, a one-sided sequence or a word $F_W(x)$ over $W$ such that $x=\prod_{i} (F_W(x))_i$ (here, the product means the concatenation of words). A  $W$-factorisation of a two-sided sequence  $x$ is a two-sided sequence  $F_W(x)=(w_i)_{i\in\Z}$ over $W$ such that $x=\prod_{i} (F_W(x))_i$, and $w_0=x_{[n,n+|w_0|)}$, $w_{-1}=x_{[n-|w_{-1}|,n)}$ for some $n\geq 0$, $n-|w_{-1}|<0$; if $n=0$, we say that the $W$-factorisation $F_W(x)$ is \textit{centred}.

 Let $\varphi\colon\mathcal{A}\to\mathcal{A}^+$ be a  substitution and let $W\subset \mathcal{A}^+$ be a \emph{code} i.e.\ a finite set of words for which the factorisation over 
$W$ is unique whenever it exists.
 We say that $W$ is \textit{compatible} with $\varphi$ if for each $w\in W$, $\varphi(w)$ is factorizable over $W$. If $W$ is compatible with $\varphi$ and $x$ is a fixed point of $\varphi$ admitting some (centred in the two-sided case) $W$-factorisation $F_W(x)=\prod_i w_i$, then for each $w\in W$, there exists a unique factorisation $F_W(\varphi(w))$ such that $F_W(x)=\prod_{i} F_W(\varphi(w_i))$, and we may define the  substitution $\tau\colon W\to W^*$ by $\tau(w)=F_W(\varphi(w))$. In general, the substitution $\tau$ depends  on the choice of the $W$-factorisation $F_W(x)$ of $x$ if $x$ admits more than one $W$-factorisation. We say that $\tau$ is the substitution \textit{induced} by the $W$-factorisation $F_W(x)$. Note that $\tau(F_W(x))=F_W(\varphi(x))$, $\tau$ is primitive whenever $\varphi$ is primitive, and $\tau^n$ corresponds to the substitution $\varphi^n$ (keeping the same $W$-factorisation $F_W(x)$ of $x$). Furthermore, we have 
${}^{t}(|\varphi^n(w)|)_{w\in W}={}^{t}(|w|)_{w\in W}M^n_{\tau}$ for each $n\geq 0$.
% (which, in general, depends on the choice of a $W$-factorisation $F_W(x)$). 
\begin{enumerate}
\item (\textbf{Trivial factorisation.}) Let $\varphi\colon\mathcal{A}\to\mathcal{A}^+$ be a substitution, let $x$ be a fixed point of $\varphi$, and let $W=\mathcal{A}$. Then $x$ admits a (trivial and obviously unique) factorisation over $\mathcal{A}$ and $\tau=\varphi$.
\item (\textbf{Return words.}) Let $x$ be a (one-sided or two-sided) sequence over $\mathcal{A}$, and let $a\in \mathcal{A}$.  A word $w\in\mathrm{L}(x)$ is called \textit{a return word} to $a$ (in $x$) if  $w$ starts with $a$, $w$ has exactly one occurrence of $a$, and $wa\in\mathrm{L}(x)$.  Let $\mathrm{R}_a$ be the set of return words to $a$ in $x$; the set $\mathrm{R}_a$ is a $\Z$-\textit{code}, i.e.\ any two-sided sequence which is factorizable over $\mathrm{R}_a$ has exactly one $\mathrm{R}_a$-factorisation. This implies that any word or one-sided sequence factorizable over $\mathrm{R}_a$ has a unique $\mathrm{R}_a$-factorisation. If $x$ is uniformly recurrent,  then any word in $x$ appears in $x$ with bounded gaps and the set $\mathrm{R}_a$ is finite. If $\varphi\colon\mathcal{A}\to\mathcal{A}^*$ is a primitive substitution and $x$ is an admissible fixed point of $\varphi$ with $x_0=a$, then the set $\mathrm{R}_a$ of return words to $a$ in $x$ is finite and compatible with $\varphi$, and the (unique) substitution $\tau\colon \mathrm{R}_a\to \mathrm{R}_a^*$ is called the \textit{return substitution} of $\varphi$ to $a$. Furthermore for a two-sided $x$, its $\mathrm{R}_a$-factorisation is centred. This construction can be, in fact, carried out for any word $u=x_{[0,t]}$, but we will not need it, see e.g.\ \cite[Sec.\ 3.1]{DL-2018} for details.
\end{enumerate}

\begin{remark}\label{rem:return} Given a primitive substitution $\varphi\colon\mathcal{A}\to\mathcal{A}^*$ with a one-sided fixed point $x$, the set of the return words $\mathrm{R}_a$ to $x_0=a$ is easily computable \cite[Lem.\ 4]{Dur-periodicity}. For completeness we recall the details. Let $w_1$ be the first return word to $a$ which appears in $x$ (which is a prefix of $x$ and always appears in $\varphi^{|\mathcal{A}|}(a)$). The word $\varphi(w_1)$ is then uniquely factorizable over $\mathrm{R}_a$, we let $w_2$ be the first return word in $\varphi(w_1)$ which is different than $w_1$ (if it exists).  The word $\varphi(w_2)$ is then uniquely factorizable over $\mathrm{R}_a$, and we let $w_3$ be the first return word in $\varphi(w_1w_2)$, which is different than $w_1$ and $w_2$ (if it exists). We continue in this way until we get a return word $w_n$ such that the (unique) $\mathrm{R}_a$-factorisation of $\varphi(w_1 \ldots w_n)$ consists only of words contained in $\{w_1,\dots, w_n\}$.   Since the set of return words $\mathrm{R}_a$ is finite, this process will stop after a finite number  of steps $n$; in fact, it is not hard to see that $n \leq 2d^2|\varphi|^d$, where $d=|\mathcal{A}|$ and $|\varphi|=\max_a |\varphi(a)|$ (although this bound is probably far from optimal). It is easy to see that the words $w_1,\dots, w_n$ comprise the whole set $\mathrm{R}_a$.
\end{remark}

\subsection*{A sufficient condition} The first step in the proof of Theorem \ref{thm:automatic} is the following result, which generalizes Dekking's theorem \cite[Sec.\ 5, Thm.\ 1]{DK-77}; Dekking's original result corresponds to the factorisation over the letters, i.e.\ when $W=\mathcal{A}$.  The proof is virtually the same, and we include it for completeness.
\begin{theorem}\label{thm:dekking} Let $\varphi\colon\mathcal{A}\to\mathcal{A}^+$ be a substitution and let $x$ be a  (one-sided or two-sided) fixed point of $\varphi$.  Let $W\subset \mathcal{A}^+$ be a code, let  $F_W(x)$ be a (centred in the two-sided case) $W$-factorisation of $x$ and assume $W$ is compatible with $\varphi$. Let $\tau\colon W\to W^*$ be the substitution induced by the $W$-factorisation $F_W(x)$. If the vector $^{t}(|\varphi^n(w)|)_{w\in W}$ is a left eigenvector of $M_{\tau}$ for some $n\geq 0$, then  $x$ is automatic.
\end{theorem}
\begin{proof}
Let $k>0$ be the dominant eigenvalue of $M_{\tau}$. Since the eigenvector $^{t}(|\varphi^n(w)|)_{w\in W}$ has positive integer entries, it has to correspond  to the dominant eigenvalue $k$ and $k$ is an integer. Note that
\begin{equation}\label{eq:nice}{}^{t}(|\varphi^{n+1}(w)|)_{w\in W}={}^{t}(|\varphi^n(w)|)_{w\in W}M_{\tau}= {}^{t}(|\varphi^n(w)|)_{w\in W}\cdot k.
\end{equation}
Put $n_w=|\varphi^n(w)|$, $w\in W$ and consider the alphabet $\mathcal{B}=\{(w,i)\mid w\in W, 0\leq i<n_w\}$. Let $\varphi^n(W)=\{\varphi^n(w)\mid w\in W\}$, and let $\sigma\colon \varphi^n(W)\to \mathcal{B}^*$ be the map 
\[\varphi^n(w)\mapsto (w,0)(w,1)\dots(w,n_w-1),\] 
which relabels $\varphi^n(w)$ into $|\varphi^n(w)|$ distinct symbols in $\mathcal{B}$. Note that we may extend $\sigma$ to words $\varphi^m(w)$, $w\in W$, $m\geq n$ or to the fixed point $x$ of $\varphi$ using factorisations (of $\varphi^m(w)$, or $x$, respectively) with respect to the set $\varphi^n(W)$. By \eqref{eq:nice}, for each $w\in W$, we may write $\sigma(\varphi^{n+1}(w))$ as a concatenation of $n_w$ words of length $k$, that is, 
\[\sigma(\varphi^{n+1}(w))=v_0^w\dots v_{n_{w}-1}^w,\] where each $v_i^w$ is a word over $\mathcal{B}$ of length $k$. We now define a substitution $\bar{\varphi}\colon \mathcal{B}\to\mathcal{B}^*$ of constant length $k$ by
\[\bar{\varphi}\colon (w,i)\mapsto v_i^w \quad w\in W,\ 0\leq i<n_w,\] and a coding $\pi\colon\mathcal{B}\to\mathcal{A}$ by
\[\pi\colon (w,i)\mapsto \varphi^n(w)_i\quad w\in W,\ 0\leq i<n_w.\] It is easy to see that $\bar{\varphi}$ is well-defined, $\sigma(x)$ is a fixed point of $\bar{\varphi}$, and $\pi(\sigma(x))=x$. This shows that $x$ is a $k$-automatic sequence.
\end{proof}

\begin{continueexample}{ex:return} Recall that  $\varphi\colon\mathcal{A}\to\mathcal{A}^*$ is given by 
\[a\to aca,\ b\to bca, c\to cbcac,\]
the fixed point $x=acac\dots$ is $4$-automatic, and $\mathrm{R}_a=\{ac, acbc\}$. Using (the proof of) Theorem \ref{thm:dekking} with $W=\mathrm{R}_a$ and $n=0$, we will now write $x$ as a coding of a fixed point of a substitution of constant length.  Following the proof of Theorem \ref{thm:dekking}, we consider the 6-letter alphabet 
\[\mathcal{B}=\{(ac,0),(ac,1),(acbc,0), (acbc,1), (acbc,2), (acbc,3)\}=\{1,2,3,4,5,6\},\]
and write
\[\sigma(ac)=12,\ \sigma(acbc)=3456.\]
Since, we know that $|\varphi(ac)|=4\cdot |ac|$ and $|\varphi(acbc)|=4\cdot |acbc|$, we consider
\begin{equation*}\begin{split}
\sigma(\varphi(ac)) &=\sigma(ac|acbc|ac)=1234|5612,\\
\sigma(\varphi(acbc)) &=\sigma(ac|acbc|acbc|acbc|ac)=1234|5634|5634|5612.
\end{split} \end{equation*}
We can now define a substitution $\bar{\varphi}\colon \mathcal{B}\to\mathcal{B}^*$ of constant length $4$ and a coding $\pi\colon\mathcal{B}\to\mathcal{A}$ by
\begin{equation*}
  \begin{split}
    1&\mapsto 1234\\
    2&\mapsto 5612\\
    3&\mapsto 1234\\
    4&\mapsto 5634\\
    5&\mapsto 5634\\
    6&\mapsto 5612
  \end{split}
\qquad \qquad
  \begin{split}
    1&\mapsto a\\
    2&\mapsto c\\
    3&\mapsto a\\
    4&\mapsto c\\
    5&\mapsto b\\
    6&\mapsto c
  \end{split}
\end{equation*}
Since $\bar{\varphi}(1)=\bar{\varphi}(3)$, $\pi(1)=\pi(3)$, and  $\bar{\varphi}(2)=\bar{\varphi}(6)$, $\pi(2)=\pi(6)$, we can further simplify $\bar{\varphi}$ and $\pi$ by identifying letters 1, 3 and 2, 6 together:
\begin{equation*}
  \begin{split}
    1&\mapsto 1214\\
    2&\mapsto 5212\\
    4&\mapsto 5214\\
    5&\mapsto 5214
  \end{split}
\qquad \qquad
  \begin{split}
    1&\mapsto a\\
    2&\mapsto c\\
    4&\mapsto c\\
    5&\mapsto b
  \end{split}
\end{equation*}
The sequence $x$ is  then a coding by $\pi$ of the fixed point $121452\dots$ of $\bar{\varphi}$ starting with 1.
\end{continueexample}

\subsection*{A necessary condition and a proof of Theorem \ref{thm:automatic}} 

Recognizability for substitutions is a classical tool, which comes in many (slightly) different forms; we refer to \cite{BSTY} for a comprehensive reference. In this paper, we will use the \textit{ (right) unilateral recognizability} for substitutions of constant length, since it is best suited for our purposes.  Since we will need to differentiate between two-sided and one-sided systems now, it will be useful to use the following notation. Let $\varphi\colon\mathcal{A}\rightarrow\mathcal{A}^*$ be a substitution and let
\[X_{\varphi}=\{x\in\mathcal{A}^{\mathbf{Z}}\mid \textrm{ every factor of } x \textrm{ appears in } \varphi^n(a) \textrm{ for some } a\in\mathcal{A}, n\geq 0\}\]
denote the \textit{two-sided} system generated by $\varphi$, and let 
\[X^{\N}_{\varphi}=\{x\in\mathcal{A}^{\mathbf{N}}\mid \textrm{ every factor of } x \textrm{ appears in } \varphi^n(a) \textrm{ for some } a\in\mathcal{A}, n\geq 0\}\]
denote the \textit{one-sided} system generated by $\varphi$. It is well known, that for a primitive $\varphi$, $X_{\varphi}$ (resp.\ $X^{\N}_{\varphi}$)  is equal to the orbit closure of any admissible fixed point of $\varphi$. The following theorem captures the recognizability property that we will need (and which we formulate for two-sided systems only although this particular statement holds for one-sided systems as well).

\begin{theorem}[Right unilateral recognizability]\label{thm:recog} Let $\varphi$ be a primitive  substitution of constant length $k$ with an admissible two-sided fixed point $x$ and assume $X_{\varphi}$ is infinite. There exists $l>0$ such that for all $z\in X_{\varphi}$, $\varphi(z)_{[m,m+l)}=\varphi(z)_{[m',m'+l)}$ implies that $m=m' \mod k$. The minimal such $l$ is called the recognizability constant of $\varphi$ and is denoted by $R_{\varphi}$.
\end{theorem}
\begin{proof}
The fact that the claim is true with $z$ equal to the fixed point $x$ follows e.g.\ from \cite[Thm.\ 3.1]{Mosse-92}   (note that for a substitution of constant length $k$ recognizability as defined in \cite[Def. 1.1]{Mosse-92} is equivalent with the fact that  $x_{[m,m+l)}=x_{[m',m'+l)}$ implies that $m=m' \mod k$ for some constant $l>0$ big enough, and that  for constant length substitutions the first condition in Theorem \cite[Thm.\ 3.1]{Mosse-92} never holds). It is easy to see that the constant $l$, which works for the fixed point $x$, works, in fact, for all $z\in X_{\varphi}$, since, by minimality, all $z\in X_{\varphi}$ have the same language as $x$.
\end{proof}

For a primitive substitution $\varphi\colon\mathcal{A}\to\mathcal{A}^*$ of constant length $k$ with the recognizability constant $R_{\varphi}$, we will say that two words $w,v\in \mathrm{L}(X_{\varphi})$ of length $\geq R_{\varphi}$ have \emph{the same cut} (w.r.t.\ $\varphi$) in $y\in X_{\varphi}$, if $w$ and $v$ occur in $y$ at positions with the same residue modulo $k$. By Theorem \ref{thm:recog}, this is unambiguous since any $y\in X_{\varphi}$ can be written as $y=T^c\varphi(z)$ for some $c\in\Z$ and $z\in X_{\varphi}$.
We note that by a result of Durand and Leroy \cite[Thm.\ 4]{DL-17}, the recognizability constant is computable for primitive substitutions although we will not need it.  We will however use the following estimates  on the recognizability constant for powers of a substitution.

\begin{proposition}\label{prop:recog}\cite[Prop.\ 13]{DL-17}  Let $\varphi\colon\mathcal{A}\to\mathcal{A}^*$ be a primitive  substitution of constant length $k$ and assume that $X_{\varphi}$ is infinite. There exists $C>0$ such that for all $n\geq 1$ we have
\[R_{\varphi^{n}}\leq C k^n,\]
where $R_{\varphi^{n}}$ denotes the recognizability constant of $\varphi^n$. 
\end{proposition}
\begin{proof}
Durand and Leroy use yet another definition of recognizability, see \cite[Def.\ 1]{DL-17}; it is easy to see that for a substitution of constant length $k$, recognizability in the sense of  \cite[Def.\ 1]{DL-17} with constant $L$ implies (right) unilateral recognizability with constant $\leq 2(L+k)$. Hence, the claim follows from \cite[Prop.\ 13]{DL-17}.
\end{proof}
%\begin{proof}Remarks about different notions of recognizability.\end{proof}

We will also need the following recent result of the second author and Yassawi, which shows that dynamically two-sided minimal automatic systems and minimal purely automatic systems are the same \cite{MY-21}. This is not true on the level of sequences: there are automatic sequences, which are not purely automatic, the most famous example being perhaps the Golay--Shapiro sequence (known also as the Rudin--Shapiro sequence) \cite[Ex.\ 24, p. 205]{AlloucheShallit-book}.
\begin{theorem}\label{thm:MY}\cite[Thm.\ 22]{MY-21} Let $k\geq 2$ be an integer and let $X$ be a minimal two-sided $k$-automatic system. There exist $n\geq 1$ and a substitution $\varphi$ of constant length $k^n$ such that $X$ and $X_{\varphi}$ are conjugate. 
\end{theorem}

The crucial part in the proof of  Theorem \ref{thm:automatic} is to show that automaticity implies that ${}^{t}(|\varphi^s(w)|)_{w\in\mathrm{R_{a}}}$ is a left eigenvector of $M_{\tau}$. To prove this, we will first show that  ${}^{t}(|\varphi^n(w)|)_{w\in\mathrm{R_{a}}}$ is a left eigenvector of $M_{\tau}$ for some $n\geq 0$. We will then use the following simple fact to reduce $n$ to $s$ from Theorem \ref{thm:automatic}.

\begin{lemma}\label{lem:matrix} Let $M$ be a $d\times d$ matrix and let $s$ denote the size of the largest Jordan block of $M$ corresponding to the eigenvalue 0. Let $v$ be a vector of length $d$ and put $v_n=M^nv$, $n\geq 0$. If $v_n$  is an eigenvector of $M$ for some $n\geq 0$, then $v_s$ is an eigenvector of $M$.
\end{lemma}
\begin{proof}
 Let  $s$ denote the size of the largest Jordan block of $M$ corresponding to the eigenvalue $0$; using the Jordan decomposition of $M$ it is not hard to see that $\mathrm{Ker}(M^n)\subset\mathrm{Ker}(M^s)$ for all $n\geq 0$ (with equality for $n\geq s$). Assume that $v_n$ is an eigenvector of $M$ and let $\lambda$ be a scalar such that $Mv_n=\lambda v_n$. Since $v_n=M^nv$, we get that
\[M^n(Mv-\lambda v)=0,\]
and thus either $v$ is an eigenvector of $M$ and so $v_s=M^sv$ is an eigenvector of $M$, or $Mv-\lambda v$ lies in  $\mathrm{Ker}(M^n)\subset\mathrm{Ker}(M^s)$ and so $M^{s+1}v=\lambda M^sv$ and $v_s$ is an eigenvector of $M$. 
\end{proof}

We are now ready to prove Theorem \ref{thm:automatic}. We will first show the dynamical part of Theorem \ref{thm:automatic} (see Theorem \ref{thm:2automatic} below) in the two-sided case, and then deduce the one-sided case. At the end, we will deduce the rest of the claims in Theorem \ref{thm:automatic} and Corollary \ref{cor:left-proper}.

\begin{theorem}\label{thm:2automatic} Let $\varphi\colon \mathcal{A}\to\mathcal{A}^*$ be a primitive substitution, let $x$ be an admissible fixed point of $\varphi$, assume that $x$ is nonperiodic, and let $a=x_0$. Let $X$ be the system generated by $x$.  Let $\tau\colon\mathrm{R}_a\to\mathrm{R}_a^*$ be the return substitution to $a$, let $M_{\tau}$ denote  the incidence matrix of $\tau$, and let $s\geq 0$ denote the size of the largest Jordan block of $M_{\tau}$ corresponding to the eigenvalue $0$.  The following conditions are equivalent:\begin{enumerate}
\item\label{thm:2automatic3} $X$ is an  automatic system;
\item\label{thm:2automatic1} $X$ has an infinite automatic system $Y$ as a (topological) factor;
\item\label{thm:2automatic2} ${}^{t}(|\varphi^s(w)|)_{w\in\mathrm{R_{a}}}$ is a left eigenvector of $M_{\tau}$.
\end{enumerate}
\end{theorem}

\begin{proof}[Proof of Theorem \ref{thm:2automatic} in the two-sided case]
The fact that \eqref{thm:2automatic3} implies \eqref{thm:2automatic1} is obvious.  Note that the orbit closure $X$ of $x$ is equal to the system $X_{\varphi}$. To show that \eqref{thm:2automatic1} implies \eqref{thm:2automatic2}, we will first show that there exist integers $k$ and  $p>q\geq 1$, such that $|\varphi^p(w)|=k^{p-q}|\varphi^q(w)|$ for all  $w\in\mathrm{R}_a$. Let
 \begin{equation}\label{eq:factormap1} \pi\colon X_{\varphi}\to Y\end{equation} 
denote a factor map onto some infinite automatic system $Y$; note that $Y$ is minimal. By Theorem \ref{thm:MY}, there exist a substitution $\bar{\varphi}\colon \mathcal{B}\to\mathcal{B}^*$ of constant length and a conjugacy
 \begin{equation}\label{eq:factormap2} \bar{\pi}\colon Y \to X_{\bar{\varphi}}.\end{equation} 
Note that $X_{\bar{\varphi}}$ is infinite. Composing maps \eqref{eq:factormap1} and \eqref{eq:factormap2}, we get the factor map 
\begin{equation}\label{eq:factormap3}\tau\colon X_{\varphi}\to X_{\bar{\varphi}}.\end{equation}
Since $X_{\bar{\varphi}}$ is infinite, by Cobham's theorem for minimal substitutive systems the dominant eigenvalues of $M_{\varphi}$ and $M_{\bar{\varphi}}$ are multiplicatively dependent  \cite[Thm.\ 14]{D-french} (see also \cite[Thm.\ 11]{Durand-dynamical} for a short ergodic-theoretic proof). Passing to some (nonzero) powers $\varphi^e$ and $\bar{\varphi}^f$ of $\varphi$ and $\bar{\varphi}$, respectively, we may assume that $M^e_{\varphi}$ and $M^f_{\bar{\varphi}}$ have the same integer dominant eigenvalue $k\geq 2$; note that this means that $\bar{\varphi}^f$ is of constant length $k$. 
By the  Curtis--Hedlund--Lyndon Theorem, the factor map $\tau$  is a (centred) sliding block-code, i.e.\ there exist an $r\geq 0$, and a map $\tau_r\colon \mathrm{L}^{2r+1}(X_{\varphi})\to \mathcal{B}$ such that $\tau(x)_i=\tau_r(x_{[i-r,i+r]})$ for $i\in\mathbf{Z}$ \cite[Thm.\ 6.2.9]{LM-book}. For a word $w=w_0\dots w_{d-1}$ in $\mathrm{L}(x)$ of length $d\geq 2r+1$, we will also use the notation
\[\tau(w)=\tau_r(w_{[0,2r+1)})\tau_r(w_{[1,2r+2)})\dots \tau_r(w_{[d-2r-1,d)})  \]

to denote the image of $w$ by $\tau$, which is a word over $\mathcal{B}$ of length $d-2r$.

Let $y=\tau(x)$. For each $n\geq 1$, let $R_n$ denote the recognizability constant of $\bar{\varphi}^{fn}$. By Proposition \ref{prop:recog}, 
\begin{equation}\label{eq:ineq1} R_n\leq Ck^n, \quad n\geq 1,\end{equation} for some constant $C>0$ independent of $n$.  For each $n\geq 1$, let $m_n$ be the smallest integer such that
\begin{equation}\label{eq:ineq2}   |\varphi^{em_{n}}(a)|\geq R_n +2r.\end{equation}
For each $n\geq 0$, there exist $y^n\in X_{\bar{\varphi}}$ and $c_n\in \Z$ such that $y=T^{c_{n}}(\bar{\varphi}^{fn}(y^n))$ (see e.g.\ \cite[Lem.\ 2.11]{BKK}). Write $x$ as a concatenation of words $\varphi^{em_{n}}(w)$, $w\in\mathrm{R}_a$, and note that each word $\varphi^{em_{n}}(w)$, $w\in \mathrm{R}_a$ starts with $\varphi^{em_{n}}(a)$.  Hence, we can write $y=\tau(x)$ as a concatenation of words
\begin{equation}\label{eq:cut}\tau(\varphi^{em_{n}}(w)\varphi^{em_{n}}(a)_{[0,2r)}), \ w\in \mathrm{R}_a
\end{equation}
(of length $|\varphi^{em_{n}}(w)|$, respectively), and, by \eqref{eq:ineq2}, all words \eqref{eq:cut} share a prefix  of length $\geq R_n$. Thus, by recognizability of $\bar{\varphi}^{fn}$  (Theorem \ref{thm:recog}), all words \eqref{eq:cut} have the same cut  in $y=T^{c_{n}}(\bar{\varphi}^{fn}(y^n))$ with respect to $\bar{\varphi}^{fn}$. Since $\bar{\varphi}^{fn}$ has constant length $k^n$,  for each $w\in\mathrm{R}_a$ we have 
\begin{equation}\label{eq:formula}|\varphi^{em_{n}}(w)|=c_w^{(n)}k^n,\quad n\geq 1
\end{equation} for some integers $c_w^{(n)}\geq 1$. Since $\varphi^e$ is a primitive substitution with dominant eigenvalue $k$, for each nonempty $u\in\mathcal{A}^*$, we have 
\begin{equation}\label{eq:PF}
\lim_{n\to\infty} \frac{|\varphi^{en}(u)|}{k^n}=c(u)
\end{equation} for some $c(u)>0$ \cite[Prop.\ 8.4.1]{AlloucheShallit-book}. Using the fact that $m_n$ is the smallest integer satisfying  \eqref{eq:ineq2}, we have that 
\[|\varphi^{e(m_{n}-1)}(a)|<R_n+2r\leq Ck^n+2r,\]  and hence, by \eqref{eq:PF} applied to $a$, $k^{m_n-n}$ is bounded independently of $n$. Applying \eqref{eq:PF}  to the words $w\in\mathrm{R}_a$, we get that the integers $c_w^{(n)}$, $n\geq 1$, are bounded independently of $n$.

Now, by pigeonhole principle, we can find two integers $p>q\geq 1$ such that  $c_w^{(p)}=c_w^{(q)}$ for all $w\in \mathrm{R}_a$. By \eqref{eq:formula}, we have that  
\begin{equation}\label{eq:final} |\varphi^{em_p}(w)|=k^l|\varphi^{em_q}(w)|,\quad w\in\mathrm{R}_a,\end{equation} where $l=p-q\geq 1$; note that this implies that $m_p>m_q$. 

 Let $\tau\colon\mathrm{R}_a\to\mathrm{R}_a$ be the return substitution to $a$ and let $M_{\tau}$ denote  the incidence matrix of $\tau$. Let $v_n={}^{t}(|\varphi^n(w)|)_{w\in\mathrm{R_{a}}}$, $n\geq 0$  and note that $v_n=v_0M_{\tau}^n$ for $n\geq 0$. Thus, we may rewrite equality \eqref{eq:final} as
\[ v_{em_{q}}M_{\tau}^{e(m_{p}-m_{q})} =  k^{l} v_{em_{q}},\]  which shows that $v_{em_{q}}$ is a left eigenvector of $M_{\tau}^{e(m_{p}-m_{q})}$ (corresponding to a nonzero eigenvalue), and thus it is a left eigenvector of $M_{\tau}$. By Lemma \ref{lem:matrix}, applied to the transposes of the  vectors $v_n$ and the matrix $M_{\tau}$, we get that $v_s= {}^{t}(|\varphi^s(w)|)_{w\in\mathrm{R_{a}}}$ is a left eigenvector of $M_{\tau}$, where  $s$ denotes the size of the largest Jordan block of $M_{\tau}$ corresponding to the eigenvalue $0$. This shows the claim.

To show  that \eqref{thm:2automatic2} implies \eqref{thm:2automatic3} assume that  ${}^{t}(|\varphi^s(w)|)_{w\in\mathrm{R_{a}}}$ is a left eigenvector of $M_{\tau}$.  By Theorem \ref{thm:dekking} (applied to  $W=\mathrm{R_{a}}$, $n=s$), $x$ is automatic. Hence $X=X_{\varphi}$ is automatic.
\end{proof}

To deduce Theorem \ref{thm:2automatic} for one-sided system, we first show the following simple lemma.
\begin{lemma}\label{lem:one-sided} Let $\varphi\colon \mathcal{A}\to\mathcal{A}^*$ be a primitive substitution, let $r\geq 0$,  and let $\pi_r\colon \mathcal{A}^{2r+1}\to\mathcal{B}$ be a block map. Let $\pi\colon \mathcal{A}^{\N}\to\mathcal{B}^{\N}$ be the map $\pi((x_n)_n)=(\pi_r(x_{[n,n+2r+1)}))_n$ induced by $\pi_r$ on $\mathcal{A}^{\N}$, and let $\pi\colon \mathcal{A}^{\Z}\to\mathcal{B}^{\Z}$ be the map  $\pi((x_n)_n)=(\pi_r(x_{[n-r,n+r]}))_n$ induced by $\pi_r$ on  $\mathcal{A}^{\Z}$ (and denoted by the same letter). Then $\pi(X^{\N}_{\varphi})$ is $k$-automatic if and only if $\pi(X_{\varphi})$ is $k$-automatic.
\end{lemma}
\begin{proof} Since $\varphi$ is primitive, we have that $\mathrm{L}(X_{\varphi})=\mathrm{L}(X^{\N}_{\varphi})$, and thus $\mathrm{L}(\pi(X_{\varphi}))=\mathrm{L}(\pi(X^{\N}_{\varphi}))$.

First assume that $\pi(X_{\varphi})$ is $k$-automatic and let $z\in \pi(X_{\varphi})$ be a $k$-automatic sequence. Write $z=z''.z'$, and note that $z'$ is a one-sided $k$-automatic sequence that lies in $\pi(X^{\N}_{\varphi})$. Since $\pi(X^{\N}_{\varphi})$ is minimal, $z'$ generates  $\pi(X^{\N}_{\varphi})$ and $\pi(X^{\N}_{\varphi})$ is $k$-automatic.

Conversely, assume that $\pi(X^{\N}_{\varphi})$ is $k$-automatic, and let $z\in \pi(X^{\N}_{\varphi})$ be a $k$-automatic sequence. Let $\spadesuit$ be a symbol not in $\mathcal{A}$ and consider a two-sided sequence $z'={}^{\omega}\spadesuit.z$, where $y={}^{\omega}\spadesuit$ denotes the constant left-infinite sequence consisting of $\spadesuit$; note that $z'$ is $k$-automatic (in fact, any left-infinite automatic sequence $y$ would do). Let $Z'=\overline{\mathcal{O}(z')}$ be the (nonminimal) system generated by $z'$, and note that  $\mathrm{L}(\pi(X_{\varphi}))\subset \mathrm{L}(Z')$. Hence, $\pi(X_{\varphi})$ is a subsystem of $Z'$. By \cite[Thm.\ 2.9]{BKK}, all subsystems of a $k$-automatic systems are $k$-automatic, and so  $\pi(X_{\varphi})$ is $k$-automatic.
\end{proof}

\begin{proof}[Proof of Theorem \ref{thm:2automatic} in the one-sided case] The fact that \eqref{thm:2automatic3} implies \eqref{thm:2automatic1} is obvious and the implication from \eqref{thm:2automatic2} to \eqref{thm:2automatic3} can be shown in the same way as in the two-sided case using the one-sided version of Theorem \ref{thm:dekking}. To show that \eqref{thm:2automatic1} implies \eqref{thm:2automatic2}, assume that $X^{\N}_{\varphi}$ has an infinite topological factor $\pi(X^{\N}_{\varphi})$. The factor map $\pi$ is induced by some block map $\pi_r\colon \mathcal{A}^{2r+1}\to\mathcal{B}$, which gives the factor map $\pi\colon X_{\varphi}\to \pi(X_{\varphi})$ as in Lemma \ref{lem:one-sided}. Then, $\pi(X_{\varphi})$ is an infinite factor  of $X_{\varphi}$.  By Lemma \ref{lem:one-sided}, $\pi(X_{\varphi})$ is automatic and hence  \eqref{thm:2automatic2} is satisfied using the corresponding claim from the two-sided case.
\end{proof}

\begin{proof}[Proof of Theorem \ref{thm:automatic}] The claims \eqref{thm:automatic3}, \eqref{thm:automatic1}, \eqref{thm:automatic2} are equivalent in both one-sided and two-sided case by Theorem \ref{thm:2automatic}. We only have to show the equivalence of \eqref{thm:automatic0} and \eqref{thm:automatic0.5} with the rest of the claims.  The proof now is verbatim the same for the one-sided and two-sided case, so we do not make the distinction.

    Clearly, \eqref{thm:automatic0} implies \eqref{thm:automatic0.5}. If $x$ is a fixed point of $\varphi$ and $\varrho\colon \A \to\B$ is a coding such that $y=\varrho(x)$ is automatic and nonperiodic, then $Y=\overline{\O(y)}$ is an automatic infinite topological factor of $X=\overline{\mathcal{O}(x)}$. Hence, \eqref{thm:automatic0.5} implies \eqref{thm:automatic1}. By Theorem \ref{thm:dekking}, the criterion in  \eqref{thm:automatic2} implies that $x$ is automatic, i.e.\ \eqref{thm:automatic0} holds which finishes the proof.
\end{proof}

\begin{proof}[Proof of Corollary \ref{cor:left-proper}] The fact that $\eqref{cor:proper2}$ implies $\eqref{cor:proper1}$ follows from Theorem \ref{thm:dekking} (using the factorisation with respect to  $W=\mathcal{A}$) and the fact that codings of automatic sequences are automatic. To show the other implication, let $x$ be a fixed point of $\varphi$ and assume that $x$ is automatic. Let $a=x_0$, and let $\mathrm{R}_a$ be the set of return words to $a$ in $x$. By Theorem \ref{thm:automatic}, there exist $n\geq 0$ and $k\geq 2$ such that $|\varphi^{n+1}(w)|=k|\varphi^n(w)|$ for $w\in \mathrm{R_a}$. Since $\varphi$ is left-proper, each word $\varphi(b)$, $b\in\mathcal{A}$ starts with $a$, and so each word $\varphi(b)$ is (uniquely) factorisable over $\mathrm{R}_a$. Hence, $|\varphi^{n+2}(b)|=k|\varphi^{n+1}(b)|$ for all $b\in\mathcal{A}$ and ${}^{t}(|\varphi^{n+1}(a)|)_{a\in\mathcal{A}}$ is a left eigenvector of $M_{\varphi}$. By Lemma \ref{lem:matrix}, ${}^{t}(|\varphi^s(a)|)_{a\in\mathcal{A}}$ is a left eigenvector of $M_{\varphi}$, where $s$ denotes the size of the largest Jordan block of $M_{\varphi}$ corresponding to the eigenvalue $0$.
\end{proof}

\section{Fibres over k-adic odometers}\label{sec:spectral}

	In this section we show that there exists a minimal substitution system with $\Z_2$ as the maximal equicontinuous factor with the factor map being everywhere uncountable-to-one and provide motivation for Conjecture \ref{con:dynamic_char}. For simplicity we restrict ourselves here to two-sided systems.

	We will use a little bit of graph-theoretic terminology. 	
	A \emph{tree} is an undirected graph in which every pair of distinct vertices is connected by exactly one path.
	A \emph{rooted tree} is a tree in which one vertex has been designated  as the root.
	In a rooted tree, the \emph{parent} of a vertex $v$ is the vertex connected to $v$ on the path to the root. 	A child of a vertex $v$ is a vertex of which $v$ is the parent. 
	A \emph{descendant} of a vertex $v$ is any vertex that is either a child of v or is (recursively) a descendant of a child of v. 
	 A \emph{leaf} is a vertex with no children.
	A disjoint union of trees is called a \emph{forest}.
	Equivalently, a forest is an undirected graph in which every pair of distinct vertices is connected by at most one path.

\begin{theorem}\label{prop:example_factortokadics}
Let $\varphi\colon \{a,b\}\to\{a,b\}^*$ be a substitution given by
\[a\mapsto abbbba,\ b\mapsto aa.\]
The system $X=X_{\varphi}$ has $\Z_2$ as the maximal equicontinuous factor and all the fibres of a factor map $\pi\colon X\to \Z_2$ are uncountable.
\end{theorem}
\begin{proof}

	The fact that $\Z_2$ is the maximal equicontinuous factor of $X$ was shown in Example \ref{ex:spectral}.
	To show the last claim  we first need to explicitly define the factor map $\pi\colon X\to \Z_2$.

	Put $\A=\{a,b\}$.
	By the recognizability of $\varphi$ in $X = X_{\varphi}$, the sets
\begin{equation}\label{eq:recognisability}
\P_n=\{ T^i\varphi^n[c] \colon c\in\A \mbox{ and } 0\leq i < \abs{\varphi^n(c)}\},\quad n\geq 1
\end{equation}
	form  a sequence of nested clopen partitions of $X$. 		
	All points in $X$ are in one-to-one correspondence with infinite itineraries given by elements of $\P_n$, $n\geq 1$ \cite{CanteriniSiegel2001}. (In general, by \cite{CanteriniSiegel2001}, for any primitive aperiodic substitution $\varphi$, all points in $X_{\varphi}$ are in one-to-one correspondence with infinite itineraries given by elements of $\P_n$ with the possible exception of the shifts of the $\varphi$-periodic points. 
		Since our substitution $\varphi$ has only one two-sided $\varphi$-periodic point---the fixed point ${}^{\infty}\varphi(a)\varphi^{\infty}(a)$)---all sequences in $X$ are faithfully represented.)

	By the calculations in Example \ref{ex:spectral},
    \begin{equation}\label{eq:growthofletters}
\begin{split}
\abs{\varphi^n(a)} &= 2^n f_n(a) = 2^n\left(\frac{(-1)^{n+1} + 4\cdot 2^n}{3}\right), \\
\abs{\varphi^n(b)} &= 2^n f_n(b) = 2^n\left(\frac{(-1)^{n} + 2\cdot 2^n}{3}\right).
\end{split}
\end{equation}

In particular, both $\abs{\varphi^n(a)}$ and $\abs{\varphi^n(b)}$ are divisible by $2^n$ for all $n\geq 1$.
	Hence, for each $n\geq 1$ the sets
\[
X(n,j)=\bigcup \{ T^i \varphi^n[c]\mid c\in \A, \, 0\leq i < \abs{\varphi^n(c)}, \, i\equiv j \bmod 2^n\}, \quad 0\leq j<2^n
\] 
form the $T^{2^{n}}$-\emph{minimal (cyclic) partition} of $X$: for each $n\geq 1$, $X(n,j)$, $0\leq j<2^n$ is a minimal system under the map $T^{2^{n}}$ and $T(X(n,j)) = X(n,j+1)$, where $j+1$ should be understood modulo $2^n$, see e.g.\ \cite[Def.\ 1 \& Lem.\ 3]{DK-77} for details.
	One has $X(n+1,j)\subset X(n,j')$ if and only if $j\equiv j' \bmod 2^n$.
 			Hence, the  map $\pi\colon X\to \Z_2 = \varprojlim \Z/2^n\Z$ given by
 			\[\pi(x) = (z_n)_n\in \Z_2  \text{ if and only if } x\in X(n,z_n) \text{ for each } n\geq 1 \]
 			is a well-defined factor map.
 	One can identify each $z\in\Z_2$ with its unique $2$-adic expansion $(l_n)_{n\geq 0}\in \{0,1\}^{\N}$ via the equation $z = \sum_{n=0}^{\infty} l_{n}2^n$.
 	For each $x\in X$ and $n\geq 0$, define  $l_n$  to be  the  $n$th term in the base-2 expansion of $i$, where $i$ is such that $x\in X(n+1,i)$. 
  	 Then $(l_n)_{n\geq 0}$ is the $2$-adic expansion of $\pi(x)=z$.
	 
	 One can conveniently visualise the fibres of the factor map $\pi$ by the paths in a labelled forest $T_{\varphi}$ defined from $\varphi$ as follows (the reader may first look at Figure \ref{fig:bigree} and Table \ref{tab:sequence_lengths} to get the idea).
	 The vertices of the forest $T_{\varphi}$ at level (or depth) $n\geq 0$ are given by the cylinder sets 
	 \[T^i\varphi^{n+1}[c]\quad \text{for}\quad 0\leq i<\abs{\varphi^{n+1}(c)} \text{ and } c\in\A=\{a,b\}.\]
	 Since, by \eqref{eq:growthofletters}, $\abs{\varphi^n(a)} + \abs{\varphi^n(b)} = 2\cdot 4^n$ for each $n\geq 1$,  our forest $T_{\varphi}$ has $2\cdot 4^{n+1}$ vertices at level $n$.
	 In particular, it has $8$ vertices at the 0th level; it consists of 8 disjoint trees rooted at $T^i\varphi[c]$, $0\leq i<\abs{\varphi(c)}$, $c\in\A$.
	 
	For each $n\geq 1$, the node $T=T^i\varphi^{n+1}([c])$ in level $n$ is a child of the node $T^{i'}\varphi^n([c'])$ in level $n-1$  if and only if 
	$T^i\varphi^{n+1}[c]\subset T^{i'}\varphi^n[c']$. 
	 More combinatorially, one can see the parent-child relationship in $T_{\varphi}$ as follows:
	 For letters $c,d\in \A$ let $P_c(d)$ be the set of all prefixes $w$ of $\varphi(c)$ such that $wd$ is also a prefix of $\varphi(c)$.
	 In our case we have
	\begin{align*}
	P_a(a) &= \{\epsilon, abbbb\}      & P_a(b) &= \{a, ab, abb, abbb\} \\
	P_b(a) &= \{\epsilon, a\}         & P_b(b) &= \emptyset.
\end{align*}
		 Looking at the equations
		 	 \begin{equation}\label{eq:powern}
	 \varphi^{n+1}(a)= \varphi^{n}(a)\varphi^{n}(b)\varphi^{n}(b)\varphi^{n}(b)\varphi^{n}(b)\varphi^{n}(a)    \quad \mbox{and}\quad \varphi^{n+1}(b) = \varphi^{n}(a)\varphi^{n}(a),
	 \end{equation}
	we see that, for each $n\geq 1$,  the children of $T^i\varphi^n[c]$ are given by $T^{i+\abs{\varphi^n(v)}}\varphi^{n+1}[d]$, $d\in \A$, $v \in P_c(d)$.	 
	 
	Now we label the nodes in $T_{\varphi}$ with colours $0$ and $1$: Each node $T^i\varphi^n([c])$, $n\geq 1$, has label $l\in\{0,1\}$ given by the $(n-1)$-th term in the base-2 expansion of $i$.
	Figure \ref{fig:bigree} shows  (parts of) the first 4 levels of the tree in $T_{\varphi}$ starting at the node $\varphi[a]$; it should be read together with Table \ref{tab:sequence_lengths} and Equation \ref{eq:powern}.
	
	\begin{table}[htbp]
\centering
\caption{Lengths of the iterated letters for $n=1, 2, 3$.}
\label{tab:sequence_lengths}
\begin{tabular}{c|cc}
\hline
\textbf{n} & \textbf{$|\phi^n(a)|$} & \textbf{$|\phi^n(b)|$} \\
\hline
1 & 6   & 2   \\
2 & 20  & 12  \\
3 & 88  & 40  \\
\hline
\end{tabular}
\end{table}
	
\begin{figure}[htbp]
    \centering
    \includegraphics[width=\textwidth]{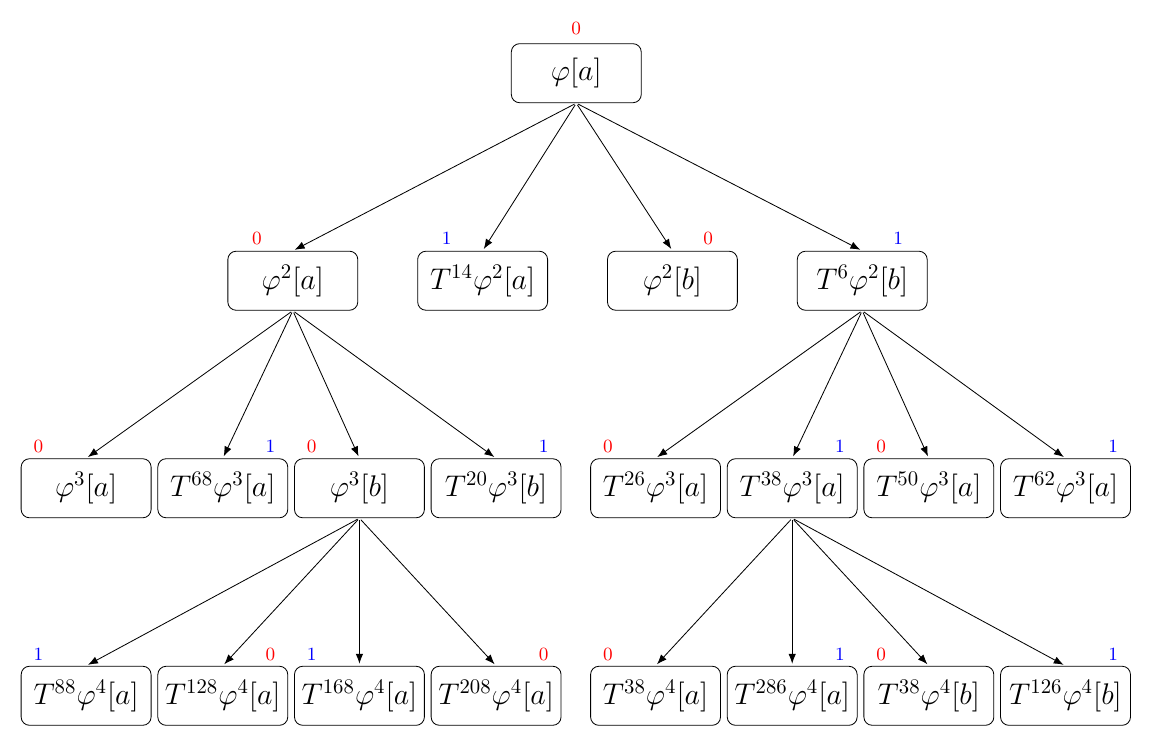}
    \caption{The first four levels of the tree in $T_{\varphi}$ rooted at $\varphi[a]$, showing 8 out of 16 nodes in level 2 and 8 out of 64 nodes in level 3.}
    \label{fig:bigree}
\end{figure}

	The points in $X_{\varphi}$ are in one-to-one correspondence with the infinite paths in $T_{\varphi}$. 
	 By the definition of the factor map $\pi\colon X_{\varphi}\to \Z_2$, the 0-1 sequence $(l_n)_{n\geq 0}$ labelling the path corresponding to $x\in X_{\varphi}$ is the $2$-adic expansion of $\pi(x)\in\Z_2$; i.e.\ $\pi(x) = \sum_{n=0}^{\infty} l_n2^n$. 
	 Clearly, for $z=\sum_{n=0}^{\infty} l_n2^n \in \Z_2$, the size of the fibre $\pi^{-1}(z)$ corresponds to the number of infinite paths in $T_{\varphi}$ labelled by $(l_n)_{n\geq 0}$.
	 To show that $\pi^{-1}(z)$ is uncountable it is thus enough to show that $T_{\varphi}$ has uncountably many infinite paths labelled by $(l_n)_{n\geq 0}$. 
	 We will show---as Figure \ref{fig:bigree} suggests---that each node in $T_{\varphi}$ has exactly 4 children: two labelled 0 and two labelled 1. 
	 This easily implies that for any labelling sequence $l=(l_n)$ the subforest of $T_{\varphi}$ given by all infinite paths labelled by $l$ comprises exactly 4 full binary trees and gives the desired claim; see also Figure \ref{fig:tree0}.
	 
	 \begin{figure}[htbp]
    \centering
    \includegraphics[width=\textwidth]{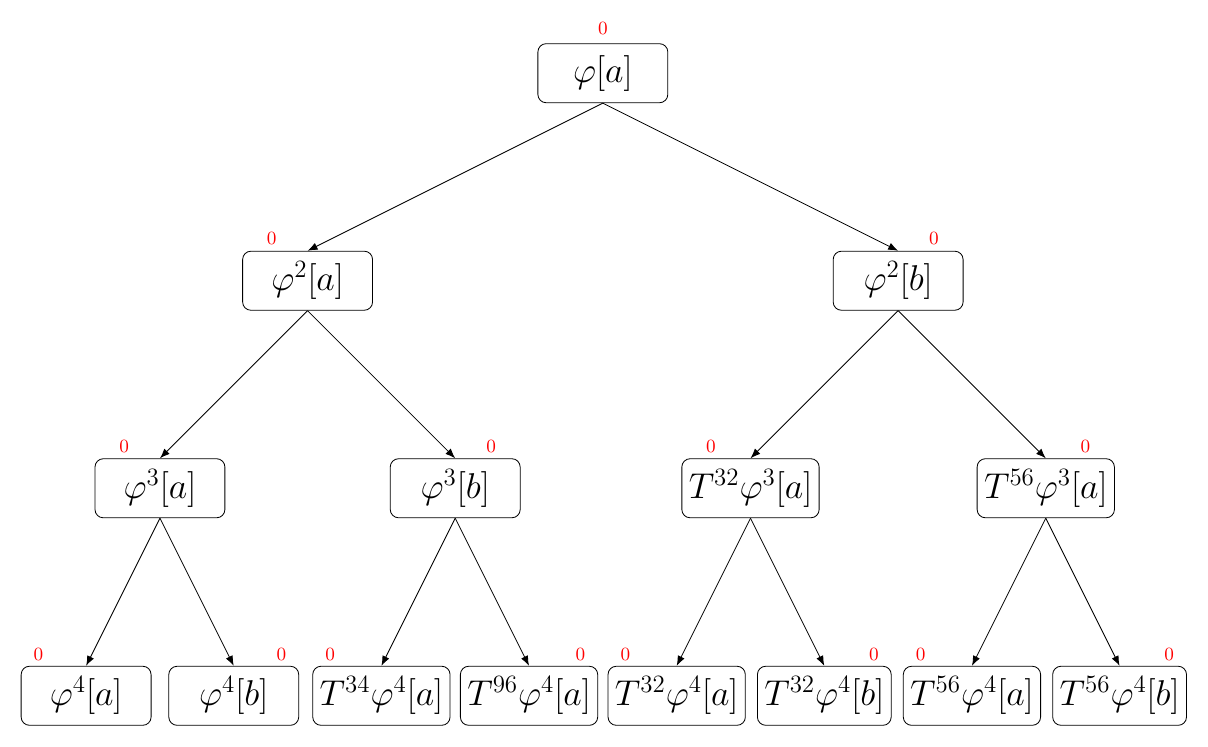}
    \caption{The complete first 4 levels of the subtree of $T_{\varphi}$ starting with $\varphi[a]$ corresponding to the label $l=0^{\infty}$. }
    \label{fig:tree0}
\end{figure}
	
	We denote by $\shortoverline{\hspace{0.4cm}}\colon \{0,1\}\to \{0,1\}$ the operation sending 0 to 1 and 1 to 0. 
	Fix $n\geq 1$ and  $i\in\N$. Write $i=\sum_{k=0}^{\infty}l_k2^k$ for the base-2 expansion of $i$; in particular, let $l_{n}$ be the $n$th term in base-2 expansion of $i$.
	The proof is now basically contained in Figures \ref{fig:treea} and \ref{fig:treeb} below.

	\begin{figure}[htbp]
    \centering
    \includegraphics[width=\textwidth]{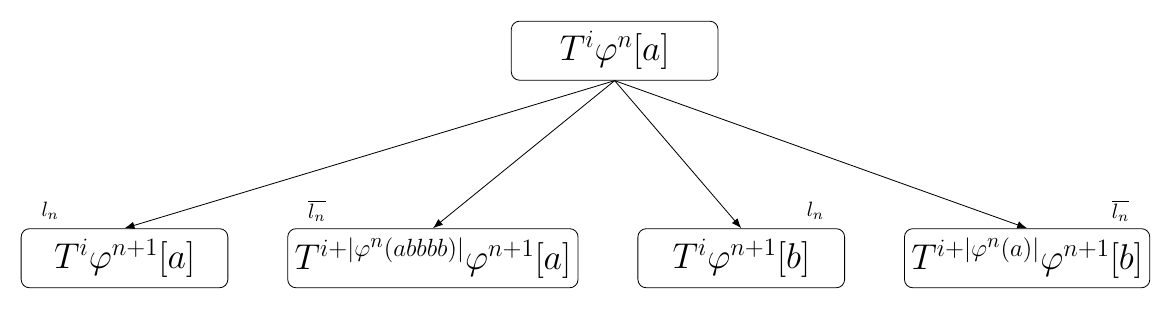}
    \caption{The children of $T^i\varphi^n[a]$ for $0\leq i<\abs{\varphi^n(a)}$ with their labels.}
    \label{fig:treea}
\end{figure}

\begin{figure}[htbp]
    \centering
    \includegraphics[width=\textwidth]{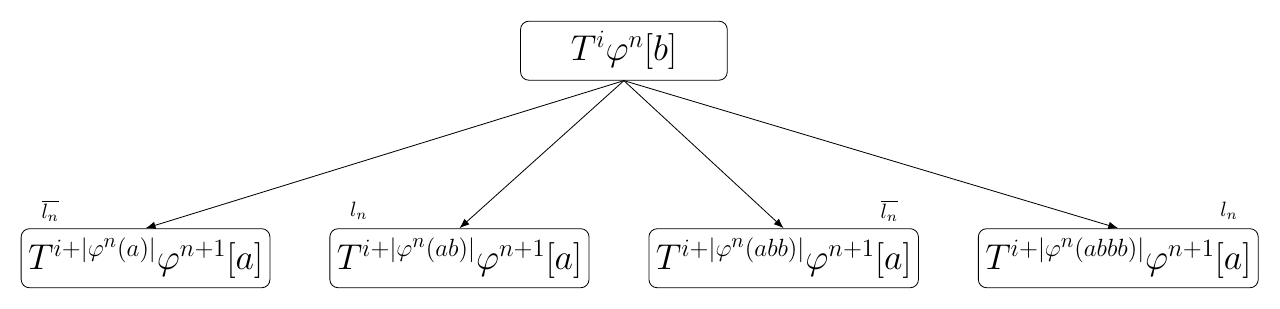}
    \caption{The children of $T^i\varphi^n[b]$ for $0\leq i<\abs{\varphi^n(b)}$ with their labels.}
    \label{fig:treeb}
\end{figure}

	The parent-child relationships in Figures \ref{fig:treea} and \ref{fig:treeb} follow from the equations \eqref{eq:powern}.
	The labellings in Figures \ref{fig:treea} and \ref{fig:treeb} follow easily from the fact that, by Equation \eqref{eq:growthofletters}, for any $n\geq 1$, both numbers $\abs{\varphi^n(a)}$ and $\abs{\varphi^n(b)}$  are divisible by $2^n$ and not divisible by $2^{n+1}$, i.e.\ they have the base-2 expansions of the form
	\begin{equation}\label{eq:nice_exp}
	0\cdot 2^0 + \dots + 0\cdot 2^{n-1} +1\cdot 2^n + \text{ higher terms}.
	\end{equation}
	For example, to see that the second node $T=T^{i + \abs{\varphi^n(ab)}}\varphi^{n+1}[a]$ in Figure \ref{fig:treeb} is labelled $l_n$ recall that the label of $T$ is given by the $n$th term in the base-2 expansion of $i + \abs{\varphi^n(ab)}$ and note that
	\begin{align*}
i &= l_0\cdot2^0 +l_1\cdot2^1 +\dots + l_{n-1}\cdot2^{n-1} +l_n\cdot2^n + \text{ higher terms} \\
\abs{\varphi^n(ab)} &= 0\cdot 2^0 + 0\cdot 2^1 +\dots +0\cdot 2^{n-1} + 0\cdot 2^n + \text{ higher terms},
\end{align*}

	where the second line follows from \eqref{eq:nice_exp} and the fact that $\abs{\varphi^n(ab)} = \abs{\varphi^n(a)} + \abs{\varphi^n(b)}$. 
	Adding the two equation, we see that the label of $T$  is $l_n$.
\end{proof}

	In the forest $T_{\varphi}$ from Theorem  \ref{prop:example_factortokadics}, for each $t\geq 1$ and each sequence $l=(l_n)_{n=0}^{t-1}\in \{0,1\}^{*}$ of length $t$, the number $n(t)$ of paths of length $t$ coloured by $l$ is dependent only on $t$ and equals \[n(t) = (1/2)^t\left( \abs{\varphi^t(a)} + \abs{\varphi^t(b)}\right) = 2^{t+1}.\]
	In particular, $T_{\varphi}$ has the property (\textbf{P}) that $n(t)$ grows exponentially with $t$.
	It is instructive to see an example of a substitution $\varphi$ with $\Z_2$ as the maximal equicontinous factor, whose forest $T_{\varphi}$ still has this property, but has only finitely many infinite paths labelled by some sequence $(l_n)\in \{0,1\}^{\N}$. 
	 One can construct such examples artificially from minimal constant length substitution systems whose factor map to the maximal equicontinuous factor is well-known to be everywhere finite-to-one.
	 Below  $\nu_2(n)$ denotes the 2-\emph{adic valuation} of an integer $n\in\N$; i.e.\ the biggest $m$ such that $n$ is divisible by $2^m$.

\begin{example}\label{ex:Thue-Morse_tree}
	Let
	\[\tau(a) = abba\quad \tau(b) = baab\] 
be the the square of the usual Thue--Morse substitution and consider the  system $X_{\tau}$ generated by $\tau$.
	Clearly, 
	 \[\abs{\tau^n(c)} = 4^n = 2^n\cdot f_c(n) = 2^n\cdot 2^n, \quad c\in\A;\]
	 compare with \eqref{eq:growthofletters}.
	 Again, $\tau$ is recognizable in $X_{\tau}$: we have the sequence of nested partitions
	 \[
\P_n=\{ T^i\varphi^n[c] \colon c\in\A \mbox{ and } 0\leq i < 4^n\},\quad n\geq 1.
	 \]
	 Here the map associating with each $x\in X_{\varphi}$ the itinerary $(P_n)_n$ with $P_n\in\P_n$ such that $x\in \bigcap P_n$ is injective on all points except the shifts of the fixed points \cite{CanteriniSiegel2001}.
	 Indeed, both one-sided fixed points $\tau^{\infty}(a)$ and $\tau^{\infty}(b)$ have two prolongations to the left in $X_{\tau}$, since all words $aa, ab, ba, bb$ lie in the language of $X_{\tau}$.
	 Thus, both itineraries   $(\tau^n[a])_{n\geq 1}$ and $(\tau^n[b])_{n\geq 1}$ (and their shifts) correspond to exactly two points in $X_{\varphi}$.

	 Now, one constructs forest $T_{\tau}$ in exactly the same way as in Theorem  \ref{prop:example_factortokadics}; see also Figure \ref{fig:treeThue} and Remark \ref{rem:tree_general} below.
     	\begin{figure}[htbp]
    \centering
    \includegraphics[width=\textwidth]{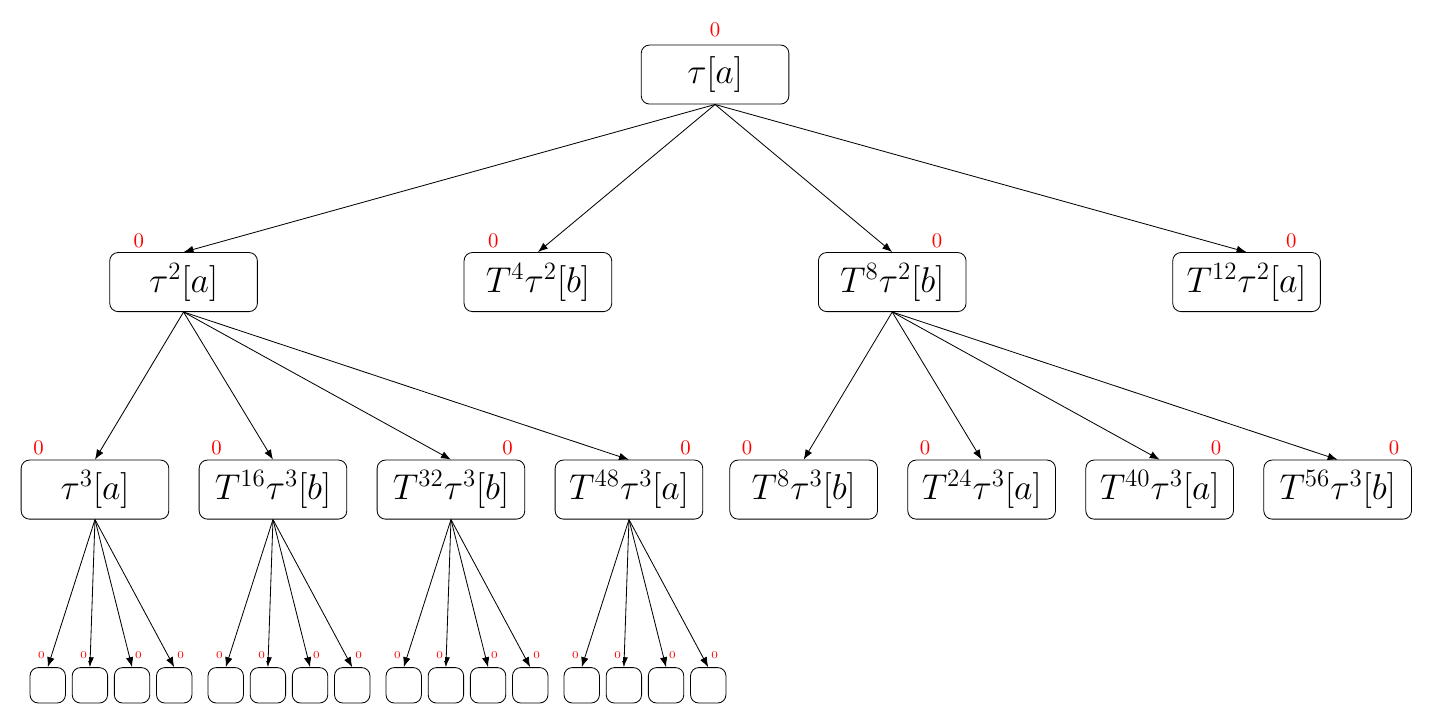}
    \caption{The complete first 4 levels of the tree with labels 0 starting with $\tau[a]$ for the Thue--Morse substitution $\tau$. Here the nodes $T^4\tau^2[b]$, $T^{12}\tau^2[a]$, $T^8\tau^3[b]$, $T^{24}\tau^3[a]$, $T^{40}\tau^3[a]$, $T^{56}\tau^3[b]$ are \emph{leaves}, i.e.\ they have no continuation with label 0.}
    \label{fig:treeThue}
\end{figure}
	 We will show that there are only two paths in $T_{\tau}$ coloured by $0^{\infty}$: $(\tau^n[a])_{n\geq 1}$ (which corresponds to two different fixed points)  and $(\tau^n[b])_{n\geq 1}$ (which corresponds to two other fixed points).
	 This of course aligns with what is known about the two-sided Thue--Morse system: the 0 fibre of its standard factor map to $\Z_2$ has 4 elements.
	 
	 As in Theorem  \ref{prop:example_factortokadics},  $T_{\tau}$ has $2\cdot 4^{n+1}$ nodes on level $n\geq 0$: half coloured by 0 and half coloured by 1.
	 The nodes at level $n\geq 0$ such that the path leading to them from the root is only coloured by 0,   are given exactly by $T^{i}\tau^n[c]$, $i\equiv 0 \bmod 2^n$, $c\in\A$; their number grows exponentially with $n$.

	It is not difficult to see that if a node $T^{i'}\tau^m[c']$, $0\leq i'<4^m$, $c'\in\A$ is a descendant of the  node $T^i \tau^n[c]$, $0\leq i<4^n$, $c\in\A$, then  $i' \equiv i \bmod 4^n$. 
	Thus, for $i \neq 0$, all descendants of $T^i \tau^n[c]$, $i\equiv 0 \bmod 2^n$ lying in level $\nu_2(i)$  are labelled by 1, since we necessarily have $\nu_2(i') = \nu_2(i)$. 
	Hence, starting from $T^i \tau^n[c]$ with $i\equiv 0 \bmod 2^n$, $i\neq 0$ all paths labelled only by 0 eventually terminate.
	Clearly, both $(\tau^n[a])_{n\geq 1}$ and $(\tau^n[b])_{n\geq 1}$ are valid infinite paths labelled only by 0, which shows the claim.

\end{example}

\begin{remark}\label{rem:tree_general}
	Let $\varphi\colon\A\to\A^*$ be a primitive \emph{left-proper} substitution. Assuming that $\varphi$ is left-proper is no loss of generality, since any minimal substitution system is conjugate to one generated by a left-proper substitution  \cite[Prop.\ 31]{D-2000}.
	In the following, for simplicity, we restrict ourselves only to systems $X_{\varphi}$ that have odometers of the form $\Z_p$ with $p$ prime as an equiconinuous factor; note that $\Z_p = \Z_{p^m}$ for any $m\geq 1$.
	
	Let $p$ be a prime.
	 A known criterion states that  $\Z_p$ is a factor of $X_{\varphi}$ if and only if  for each $l\in\N$ all sequences $(\abs{\varphi^n(a)})_n$, $a\in\A$ are eventually divisible by $p^l$, see e.g. \cite[Lem.\ 28]{D-2000}.
If $\Z_p$ is a factor of $X_{\varphi}$, then one can further show that---after possibly changing  $\varphi$ to its power---we have
	\begin{equation}\label{eq:forthetree}
	\abs{\varphi^n(a)} = q^{n-c} f_a(n),\quad a\in\A, \ n\geq 1,
	\end{equation}
	where $q=p^m$ for some $m\geq 1$, $f_a: \N \to \N$, and $c\geq 0$.
	This allows us to define the factor map $\pi\colon X_{\varphi}\to\Z_q$ analogous to the one in the proof of Proposition~\ref{prop:example_factortokadics}:
    indeed, for each $n\geq 1$ the sets
\[
X(n,j)=\bigcup \{ T^i \varphi^{n+c}[a]\mid a\in \A, \, 0\leq i < \abs{\varphi^{n+c}(a)}, \, i\equiv j \bmod q^n\}, \quad 0\leq j<q^n
\] 
form the $T^{q^{n}}$-\emph{cyclic partition} of $X_{\varphi}$.
    Furthermore, after taking some power of $\varphi$ again (and changing $q$ to the highest possible  power of $p$), if necessary, one can ensure that 
	\begin{equation}\label{eq:nice_form}
	\text{there exists } a\in\A \text{ such that } f_a(n)\neq 0 \bmod q \text{ for all } n\geq 1.
	\end{equation}

	Once \eqref{eq:forthetree} and \eqref{eq:nice_form} hold, by Theorem \ref{thm:automatic}, the fact that $X_{\varphi}$ is not hidden $p$-automatic is equivalent to the following statement
	\begin{equation}\label{eq:not_hidden_automatic}
	\text{the sequence } (f_a(n))_n \text{ is unbounded for some } a\in \A.
	\end{equation}
	 For a substitution $\varphi$ satisfying \eqref{eq:forthetree}, one defines the corresponding forest $T_{\varphi}$ with labels in $\{0,\dots, q-1\}$ in the analogous way as before: The nodes at level $n\geq 0$ are given by $T^i\varphi^{n+1+c}[a]$, $0\leq i<\abs{\varphi^{n+1+c}(a)}$, $a\in \A$ with $T^i\varphi^{n+1}([c])$, $n\geq c+1$ being  a child of the node $T^{i'}\varphi^{n}([c'])$ if and only if 
	$T^i\varphi^{n+1}[c]\subset T^{i'}\varphi^n[c']$.
	Each node $T^i\varphi^n[a]$, $n\geq c+1$ in level $n-c-1$ has label $l$ given by the $(n-1)$th term of the base-$q$ expansion of $i$.
	For each $x\in X_{\varphi}$ represented by some path in $T_{\varphi}$, its label $(l_n)$ gives the $q$-adic expansion of $\pi(x)\in \Z_q$.
	The unboundedness  of $f_a(n)$ in \eqref{eq:not_hidden_automatic} corresponds to (a variation of) Property (\textbf{P}). The fiber $\pi^{-1}(z)$ of some $z\in\Z_q$ is infinite if and only if there are infinitely many paths in $T_{\varphi}$ labelled by the $q$-adic expansion $l_n$ of $z$.
\end{remark}

Note that $T_{\varphi}$ depends also on the choice of $q$ in \eqref{eq:forthetree}. For the Thue--Morse substitution $\tau$ in Example \ref{ex:Thue-Morse_tree}, one can take $q=2$  or $q=4$. For $q=2$ one gets the forest as in the example; for $q=4$ one gets a forest, which doesn't satisfy Property (\textbf{P}) and, for it, it is clear that any inifinite sequence can label only finitely many paths.  In this sense, Example \ref{ex:Thue-Morse_tree} is somewhat artificial.

	One can produce examples like Example \ref{ex:Thue-Morse_tree} from any primitive $\varphi$ which is hidden automatic (characterised in Theorem \ref{thm:automatic}).
	For a primitive substitution $\varphi$ which is \emph{not hidden automatic}, but has  $\Z_q$ as a factor (i.e. for $\varphi$ satisfying \eqref{eq:forthetree}, \eqref{eq:nice_form}, and \eqref{eq:not_hidden_automatic}), the general structure of the corresponding forest $T_{\varphi}$ can be quite complicated; in particular, one does not necessarily always get nice $l$-ary trees as in Theorem  \ref{prop:example_factortokadics}. 
	 However, we were unable to find any examples of forests $T_{\varphi}$ with $\varphi$ not hidden automatic for which some label $(l_n)_{n\in \N}$ had only finitely many corresponding paths.
 In fact, we believe that this is not possible; see Conjecture \ref{con:dynamic_char}  in the introduction.

 \bibliographystyle{alpha}
 \bibliography{bibmorphic}

\end{document}